\documentclass[11pt, a4paper]{article}

\usepackage{amsfonts,amssymb,amsmath,bbm,amsthm}
\usepackage{graphicx} 
\usepackage[T1]{fontenc}
\usepackage[english]{babel}
\usepackage{color}
\usepackage{grffile}

\makeatletter \@addtoreset{equation}{section} \makeatother
\makeatletter \@addtoreset{enunciato}{section} \makeatother
\newcounter{enunciato}[section]

\newtheorem{ittheorem}{Theorem}
\newtheorem{itlemma}{Lemma}
\newtheorem{itproposition}{Proposition}
\newtheorem{itdefinition}{Definition}
\newtheorem{itremark}{Remark}
\newtheorem{itclaim}{Claim}
\newtheorem{itfact}{Fact}
\newtheorem{itconjecture}{Conjecture}
\newtheorem{itcorollary}{Corollary}

\newenvironment{theorem}{\addtocounter{enunciato}{1}
\begin{ittheorem}}{\end{ittheorem}}
\newenvironment{lemma}{\addtocounter{enunciato}{1}
\begin{itlemma}}{\end{itlemma}}
\newenvironment{proposition}{\addtocounter{enunciato}{1}
\begin{itproposition}}{\end{itproposition}}
\newenvironment{definition}{\addtocounter{enunciato}{1}
\begin{itdefinition}}{\end{itdefinition}}

\newenvironment{corollary}{\addtocounter{enunciato}{1}
\begin{itcorollary}}{\end{itcorollary}}
\begin{document}
\title{Random spanning forests,\\ Markov matrix spectra and well distributed points}

\author{\renewcommand{\thefootnote}{\arabic{footnote}}
L.\ Avena\footnotemark[1],
\renewcommand{\thefootnote}{\arabic{footnote}}
A.\ Gaudilli\`ere\footnotemark[2]}

\date{}

\footnotetext[1]{MI, University of Leiden, The Netherlands. E-mail: l.avena@math.leidenuniv.nl}
\footnotetext[2]{Aix Marseille Universit\'e, CNRS, Centrale Marseille, I2M, UMR 7373, 13453 Marseille, France. E-mail: alexandre.gaudilliere@math.cnrs.fr}

\maketitle

\begin{abstract}
This paper is a variation
on the uniform spanning tree theme.
We use random spanning forests
to solve the following problem:
for a Markov process
on a finite set of size $n$,
find a probability law
on the subsets of any given size $m \leq n$
with the property that 
the mean hitting time 
of such a random target 
does not depend
on the starting point of the random walk.
We then explore the connection
between random spanning forests
and infinitesimal generator spectrum.
In particular we give an almost probabilistic proof
of an algebraic result due to Micchelli and Willoughby~\cite{MW79}
and used by Fill~\cite{F09} and Miclo~\cite{M10}
to study the convergence to equilibrium 
of reversible Markov chains.
We finally introduce 
some related fragmentation and coalescence processes,
emphasizing algorithmic aspects, and give an extension
of Burton and Pemantle transfer current theorem~\cite{BP93}
to the non reversible case.

\vspace{0.5cm}\noindent
\end{abstract}

\newpage

%%%%%%%%%%%%%%%%%%%%% SECTION 1 %%%%%%%%%%%%%%%%%%%%%%%%%%%%%%%%%%%%%%%%
%%%%%%%%%%%%%%%%%%%%%%%%%%%%%%%%%%%%%%%%%%%%%%%%%%%%%%%%%%%%%%%%%%%%%%%%%%%%%%%%%%%%%%%%%%%%%%%%%%%%%%%%%%%%%%%%%%%%
%NEW SECTION
%%%%%%%%%%%%%%%%%%%%%%%%%%%%%%%%%%%%%%%%%%%%%%%%%%%%%%%%%%%%%%%%%%%%%%%%%%%%%%%%%%%%%%%%%%%%%%%%%%%%%%%%%%%%%%%%%%%%

\section{A random set problem and a forest measure}
\label{s1}

%%%%%%%%%%%%%%%%%%%%%%%%%%%%%%%%%%%%%%%%%%%%%%%%%%%%%%%%%%%%%%%%%%%%%%%%%%%%%%%%%%%%%%%%%%%%%%%%%%%%%%%%%%%%%%%%%%%%%%%%%%%%%%%%%%%
\subsection{``Well distributed points'' in a given graph}\label{chessboard}
Let us consider the following problem. 
We have a square chessboard with sides of length $2l$ and a simple random walk on it.
More precisely, think the chessboard as the square lattice box $\mathcal{X}=\{1, \cdots, 2l\}^2$ 
with the simple random walk $X$ on $\mathcal{X}$.
Denote by $T_R$ the hitting time of a set $R\subset\mathcal{X}$ for the walk $X$ and by $P_x$ the law of $X$ starting from $x\in\mathcal{X}$.
Can you find a probability law $\mathbb{P}$ on the subsets $R$ of $\mathcal{X}$ with cardinality $|R|=|\mathcal{X}|/2$  such that 
\begin{equation}
	\mathbb{E}\left[E_x\left[T_R\right]\right] \text{ does not depend on } x\text{ ?}
\end{equation}
In other words, can you sample $2 l^2$ ``well distributed points'' among the $4 l^2$ points of ${\cal X}$?
In this example, a possible simple answer is the following:
take $R$ to be the set of either white or black squares of the chessboard with probability $1/2$.

One could then raise the following questions:
\begin{itemize}
\item What if the random subsets $R$ are required to have any other cardinality $|R|=m \leq 4 l^2$? 
\item What if, instead of the chessboard $\mathcal{X}$, we consider any other finite weighted, possibly oriented, graph? 
\end{itemize}
We notice that the case $m = |{\cal X}|$ is trivial,
and that in the case  $m = 1$, it is known that it suffices
to choose the point in $\mathcal{X}$ according
to the stationary measure of the walk (see e.g. Lemma 10.8 in~\cite{LPW08}).
One of the main goal of this paper is to answer these questions
for $1 < m < |{\cal X}|$.

In the sequel, we work on a finite oriented weighted graph with its naturally associated continuous time Markov process.
We study a certain probability measure on the set of spanning rooted oriented forests on this graph.
It will turn out that the set of roots of the forest sampled from this measure, with conditioning on the number of roots,
provides a solution to our random set problem in full generality.
As far as practical sampling issues are concerned, by using 
an algorithm due to Wilson and Propp~\cite{W96, PW98} we can sample this measure without conditioning.
Furthermore, under the assumption that the generator of the random walk
associated with the starting weighted graph has only real eigenvalues,
we explain how to get a sample with an approximate prescribed number $m$ of roots within an error of order $\sqrt m$.
In Section \ref{setting} below, we introduce the main framework and notation, 
and in Section \ref{Paper} we describe the structure of the paper and the results we derive.

%%%%%%%%%%%%%%%%%%%%%%%%%%%%%%%%%%%%%%%%%%%%%%%%%%%%%%%%%%%%%%%%%%%%%%%%%%%%%%%%%%%%%%%%%%%%%%%%%%%%%%%%%%%%%%%%%%%%%%%%%%%%%%%%%%%
\subsection{Forest measures}\label{setting}
Let $X$ be a Markov process on a finite state space $\mathcal{X}$, with $|\mathcal{X}|=n$.
Assume $X$ is irreducible with generator given by
\begin{equation}\label{MarkovGenerator}
(Lf)(x)=\sum_{y\in {\cal X}}w(x,y)[f(y)-f(x)],\quad x\in \mathcal{X}, 
\end{equation}
with  $f:\mathcal{X}\rightarrow \mathbb{R}$ arbitrary and 
$\{ w(x,y)\in [0,\infty): (x,y)\in \mathcal{X}\times \mathcal{X}\}$ 
a given collection of non-negative transition rates.
Note that such a Markov process has variable speed depending on the current state, namely, if the chain is at position
$x$, the next jump will be performed after an exponential time of rate 
\begin{equation}\label{VariableSpeed}
w(x)=\sum_{y\in \mathcal{X}\setminus \{x\}}w(x,y)<\infty. 
\end{equation}
The collections of rates induces a structure of oriented weighted graph on $\mathcal{X}$.
In fact, consider the set of oriented edges
\begin{equation}\label{oriented edges}
\mathcal{E}=\left\{(x,y)\in \mathcal{X}\times \mathcal{X} :\: x\neq y \text{  and  } w(x,y)>0\right\}, 
\end{equation}
then $\mathcal{G}=(\mathcal{X},\mathcal{E})$ is a weighted oriented graph.

The main object of our investigation is a measure on the spanning rooted forests of ${\cal G}$.
These are the oriented subgraphs of ${\cal G}$ that can be described in the following way.
Call an ${\cal X}$-spanning {\em unrooted} forest 
any simple undirected graph without cycle
and with ${\cal X}$ as set of vertices.
The connected components of such a graph are trees.
By specifying a root, that is a particular vertex,
in each of these trees, we can define an oriented graph 
$\phi$ by orienting the edges of each rooted tree
towards its root.
If $\phi$ is a subgraph of $G$,
i.e., if each (oriented) edge of $\phi$ belongs to ${\cal E}$,
we call it a spanning rooted forest of ${\cal G}$.
We will identify $\phi$ with the collection of its edges,
that is with a subset of ${\cal E}$,
and we call $\mathcal{F}$ the set of the spanning rooted forests of ${\cal G}$.
In particular $\emptyset \in {\cal F}$
is the spanning forest made of $n$ degenerated trees
reduced to simple roots.
For $\phi\in \mathcal{F}$ we define the weight of a such a forest 
\begin{equation}\label{ForestWeight}
	w(\phi)=\prod_{e\in\phi}w(e)
	\,.
\end{equation}

\begin{definition}{\bf (Standard forest measure)}\label{TheForest}
Denote by $\rho(\phi)\subset \mathcal{X}$ the set of roots of a forest $\phi\in\mathcal{F}$.
Fix $q>0$ and define on $\mathcal{F}$ the measure $w_q$ given by
\begin{equation}\label{StandardForestMeasure}
w_q(\phi)=q^{|\rho(\phi)|}\prod_{e\in\phi}w(e)=q^{|\rho(\phi)|}w(\phi)
\,,
\quad \phi \in {\cal F}
\,.
\end{equation}
By normalizing it via the partition function
\begin{equation}\label{StandardPartitionFn}
Z(q)=\sum_{\phi\in \mathcal{F}} w_q(\phi), 
\end{equation}
we denote the resulting probability measure by 
\begin{equation}\label{StandardForestProb}
\nu_q(\phi)=\frac{w_q(\phi)}{Z(q)},\quad \phi\in{\cal F}.
\end{equation}
We call \emph{standard measure},  \emph{standard partition function} and \emph{standard probability measure}, 
the objects defined by equations \eqref{StandardForestMeasure}, \eqref{StandardPartitionFn}, and \eqref{StandardForestProb}, respectively.
\end{definition}

\medskip\par\noindent
{\bf Remarks:} In the case of symmetric rates $w(x, y)$,
there are obvious similarities between the weights appearing in equation (\ref{StandardForestMeasure})
and those of Fortuyn-Kasteleyn model. We stress here the main differences.
FK-percolation is defined on spanning graphs that are not required to be forests.
However, in the zero limit of its main parameters, properly rescaled, the model does converge
to a measure on spanning forests (see e.g.~\cite{G06}).
Nevertheless, our forests are rooted and this extra structure introduces an entropic factor
in comparison (by projection on unrooted forests) with this zero limit of FK-percolation.
To make it precise, let us call $\bar {\cal F}$ the set of unrooted spanning forests
of $\bar {\cal G} = ({\cal X}, \bar {\cal E})$, with $\bar {\cal E}$ obtained from
${\cal E}$ by identifying each edge $(x, y)$ with it `opposite' $(y, x)$.
To each rooted forest $\phi \in {\cal F}$ is associated, by construction, 
a unique unrooted forest $\bar \phi \in \bar {\cal F}$.
On the one hand, the weight associated in (the zero limit of) FK-percolation 
to each $\bar \phi \in \bar {\cal F}$ is proportional to
\begin{equation}
	w_{FK}\bigl(\bar \phi\bigr) = q^{\left|{\cal T}\left(\bar \phi\right)\right|} \prod_{\bar e \in \bar \phi} w(\bar e) 
	\,,
\end{equation}
where ${\cal T}\bigl(\bar \phi\bigr)$ is the set of the connected component
of $\bar \phi$, i.e., the set of all maximal trees contained in $\bar \phi$.
On the other hand, if $\Phi_q$ is a random forest with law $\nu_q$,
then, for each $\bar \phi \in \bar {\cal F}$,
\begin{equation}\label{toussaint}
	{\mathbb P}\left(
		\bar \Phi_q = \bar \phi
	\right)
	= \frac{1}{Z(q)} q^{\left|{\cal T}\left(\bar \phi\right)\right|}
	\prod_{\bar e \in \bar \phi} w(\bar e)
	\prod_{\bar \tau \in {\cal T}\left(\bar \phi\right)} |V(\bar \tau)|
	\,,
\end{equation}
with $V(\bar \tau)$ the set of vertices spanned by the tree $\bar \tau$
(by seeing $\bar \tau$ as a set of edges, one has $|V(\bar \tau)| = 1 + |\bar \tau|$).
There are indeed, $\prod_{\bar \tau \in {\cal T}\left(\bar \phi\right)} |V(\bar \tau)|$
ways of choosing roots for the trees of $\bar \phi$.
We get by projection a cluster size biased version of the zero limit of FK-percolation.

We also note that 
the weights in (\ref{ForestWeight}) are those associated 
by Freidlin and Wentzell
with the so called ``$W$-graphs''~\cite{FW98}
and we will recover some of their results in this paper.
Our standard forest measure has also been studied in various works.
For example to sample points from the stationary measure of the random walk~\cite{PW98},
or to study in~\cite{JRS11} recurrent configurations of Abelian dissipative sandpile
introduced in~\cite{TK00}.
This spanning forest measure and other associated objects
we will discuss later are actually a variation on the theme
of uniform spanning trees and loop-erased random walks.
We refer to~\cite{LJ11} and references therein for the vast literature
on the subject.

\medskip\par
As it will become clear, the measure in Definition \ref{TheForest} and the associated partition function encode several information related to the chain $X$.
At the occurrence, we will derive results related with slightly more general measures and partition functions. 
To this aim, we introduce here some further notation. 
Let us first introduce a natural generalization of the measure
in Definition \ref{TheForest}.
Given a collection of extra weights $\{q(x)\in[0,\infty] :\: x\in\mathcal{X}\}$, 
let $Q = (Q(x, y))_{x, y \in {\cal X}}$ be the diagonal matrix 
defined by $Q(x, x) = q(x)$ for $x \in {\cal X}$ (and $Q(x, y) = 0$ when $x \neq y$). 
We anticipate that these extra weights will be interpreted as killing rates for the chain $X$.
Set 
\begin{equation}\label{gammino}
	S=\{x\in\mathcal{X}:\: q(x)=+\infty \}
\end{equation}
and define the measure $w_Q$ by
\begin{equation}\label{ForestMeasure}
w_Q(\phi)=w(\phi)\prod_{x\in\rho(\phi)\setminus S}q(x)\mathbbm{1}_{\{S\subset \rho(\phi)\}}.
\end{equation}

By assuming that there is at least one $x \in {\cal X}$ with $q(x) > 0$
we can turn $w_Q$ into a probability measure on $\mathcal{F}$ by normalizing it via the partition function:
\begin{equation}\label{PartitionFn}
Z_Q=\sum_{\phi\in \mathcal{F}} w_Q(\phi)
\end{equation}
and we denote the resulting probability measure by 
\begin{equation}\label{ForestProb}
\nu_Q(\phi)=\frac{w_Q(\phi)}{Z_Q},\quad \phi\in{\cal F}. 
\end{equation}
This is the general form of the probability measure at the core of our investigation.

When answering the questions raised in Section \ref{chessboard}, we need the following special
case of the generalized measure in \eqref{ForestMeasure}.
For a given subset $R\subset\mathcal{X}$, suppose that the collection of extra weights $\{q(x)\in[0,\infty] :\: x\in\mathcal{X}\}$ is such that 
\begin{equation}\label{KillingR}
q(x)=
\left\{
\begin{array}{ll}
+\infty & \text{if } x \in R, \\
q \geq 0 \mbox{ ($q > 0$ if $R = \emptyset$)} & \text{otherwise}.
\end{array}
\right.
\end{equation}
In this case, $S=R$ and we write
\begin{equation}\label{Rmeasures}
w_{q, R}(\phi)= q^{|\rho(\phi)\setminus R|} w(\phi)\mathbbm{1}_{\{R\subset \rho(\phi)\}}
\,, 
\quad Z_R(q) = \sum_{\phi \in {\cal F}} w_{q, R}(\phi)
\,,
\quad \nu_{q, R}(\phi) = \frac{w_{q, R}(\phi)}{Z_{q, R}(\phi)}
\,,
\end{equation}
for the associated measure, partition function and probability, $w_Q(\cdot)$, $Z_Q$ and $\nu_Q$, respectively.
In particular, for $q=0$ and $R \neq \emptyset$, $w_{0, R}(\phi)=w(\phi)\mathbbm{1}_{\{\rho(\phi)=R\}}$ and $Z_R(0)=\sum_{\phi:\rho(\phi)=R}w(\phi)$.
Note further that, when $R=\emptyset$ (and $q > 0$),
we recover the \emph{standard} measure and partition function, $w_{q, \emptyset}(\cdot)=w_q(\cdot)$ and $Z_{\emptyset}(q)=Z(q)$.

In the sequel we will denote by $\Phi_Q$ a random variable on a probability space $(\Omega, {\mathbb P})$
with values in ${\cal F}$ and law $\nu_Q$.
We will also write $\Phi_{q, R}$ and $\Phi_q$ in the corresponding special cases.

%%%%%%%%%%%%%%%%%%%%%%%%%%%%%%%%%%%%%%%%%%%%%%%%%%%%%%%%%%%%%%%%%%%%%%%%%%%%%%%%%%%%%%%%%%%%%%%%%%%%%%%%%%%%%%%%%%%%%%%%%%%%%%%%%%%
\subsection{Results and structure of the paper}\label{Paper}

\subsubsection{Main results}
Our analysis of the forest measures introduced above will lead us to several results.
Before describing them and the organization of the paper,
we emphasize herein what we consider as our four main results.

\smallskip\par\noindent
{\bf Determinantal roots:}
First, in Theorem \ref{DeterminantalRoots} we prove that the set of
roots $\rho(\Phi_Q)$ is a determinantal process. 
In particular, denoting by
\begin{equation}
	K_q(x,y)=P_x(X(T_q)=y)
	\,,
	\quad x, y \in {\cal X}
	\,,
\end{equation}
the transition probabilities
of the Markov chain $X$ in \eqref{MarkovGenerator} observed at independent exponentially distributed times $T_q$
of parameter $q$,
we show that 
\begin{equation}
	{\mathbb P}\left(A\subset \rho\left(\Phi_q\right)\right)={\rm det}_A (K_q), \quad A\subseteq \mathcal{X},
\end{equation}
with ${\rm det}_A (K_q)$ being the determinant of $K_q$ restricted to $A$ (see Section \ref{Notation}
below for the notation).
This echoes Burton and Pemantle transfer current theorem (\cite{BP93}, Theorem 1.1).
on spanning trees associated with reversible Markov processes.
Our generator $L$ and our kernel $K_q$ are however
not required to be reversible (and they possibly have complex eigenvalues)
and we present a direct proof not relying on transfer currents.
We will actually later use our random rooted forests
to prove transfer current theorems for random rooted and unrooted spanning {\em trees}
associated with non-reversible Markov processes.

\smallskip\par\noindent
{\bf Random target:}
The second result is an answer to the questions in the introductory
Section~\ref{chessboard}. In fact, in Theorem~\ref{HittingRootsThm} we
give a formula for the hitting times of random sets constituted by the
roots of our standard random forests, with or without conditioning on
having a fixed number of roots. In particular, such a formula is
independent of the starting point $x$.
While Wilson and Propp algorithm gives a way to sample the unconditioned measure,
in Section~\ref{HittingTimes} we explain 
how to obtain, in absence of complex eigenvalues, 
a sample with approximately $m$ roots, with an error of order $\sqrt m$,
for any $m \leq n$.

\smallskip\par\noindent
{\bf Local equilibria:}
Our third result concerns the proof on an algebraic statement on
symmetric matrices derived by Micchelli and Willoughby~\cite{MW79},
Theorem 3.2 therein.  This theorem has been used in~\cite{F09, M10}
to study absorption times of reversible Markov chains
on general finite graphs, as a key tool to define, in such a general setting,
the local equilibria introduced in~\cite{DM09}.
In~\cite{M10} the author motivates the importance of having a probabilistic interpretation of
Micchelli and Willoughby algebraic result. 
In Section \ref{MW} we restate their result and give,
by means of our standard forest measure, a probabilistic presentation of their proof.

\smallskip\par\noindent
{\bf Transfer current:}
These three main results focus on the root process $\rho(\Phi_Q)$ or $\rho(\Phi_q)$
and show its deep connection with the spectrum of $L$.
In Section~\ref{archi}
we start instead, like in~\cite{C13}, from the study of the whole edge process of the spanning {\em forest} $\Phi_Q$
to give in Theorem~\ref{corvaro}
an extension of transfer current theorems for spanning {\em trees}
to the non reversible case.
This is our last main result.

\subsubsection{Organization of the paper}
The rest of the paper is organized as follows.

\smallskip\par\noindent
{\bf Background material:} Section~\ref{Background} is a warm-up section where we provide some known
background material. In Section~\ref{loops} we prove
in a slightly different way a result originally due to Marchal~\cite{M00} on loop-erased trajectories
(Proposition \ref{Marchal}). 
In Section~\ref{Wilson} we recall Wilson's algorithm
(Definition~\ref{Wilson}) and, following~\cite{M00, PW98}, 
we show how to sample our unconditioned measures (Corollary~\ref{WilsonProb}).

\smallskip\par\noindent
{\bf Results:} Section~\ref{ForRoots} presents the first analysis on our forest
measure, mainly focusing on the root process.
In Theorem~\ref{PartitionThm} and Corollary~\ref{Bernoullicity} therein
we show some connection of this measures with the spectrum of the chain $X$.
We prove that the set of roots is a determinantal process (Theorem~\ref{DeterminantalRoots})
and compute its cumulants, or truncated correlation functions (Lemma~\ref{nebbioline}).
In Section~\ref{HittingTimes} we prove Theorem~\ref{HittingRootsThm}
on the hitting times of the root set,
answering the questions raised in Section~\ref{chessboard}. 
We also compute the mean return time in the set of roots
for the Markov process started from a uniformly chosen root
(Theorem~\ref{edera}) and estimate the expected value
of the largest mean hitting time of the root set
(Theorem~\ref{blu}).
Section~\ref{MW} presents Micchelli and Willoughby~\cite{MW79}
result and proof in a probabilistic way. 
In Section~\ref{marmelata} we mention two coalescence and fragmentation
processes associated with our measures.
One of them give some information on the ``rooted partition''
induced by our spanning forests (Proposition \ref{LocalEquilibria}).
The other one is obtained
by coupling together all the standard forest
measures for different values of $q$
and raises a number of open questions.
In Section~\ref{archi} we study the full edge process
to extend classical transfer current theorems
in Theorem~\ref{corvaro}.

\smallskip\par\noindent
{\bf Appendix:} 
Four appendix sections are devoted to known results used along the paper
and that we derive in our context in order to have a self-contained work.
In Appendix~\ref{SCHUR},
we recall what is the Schur complement for block matrices
and its probabilistic interpretation (Proposition~\ref{SchurProb}).
In Appendix~\ref{Wentzell},
we give different proofs of two lemmas from Freidlin and Wentzell
(Lemmas 3.2, 3.3 in~\cite{FW98}) on hitting distribution and times of
subsets of the graph, again by analysing of our forest measure. 
This is used in Section~\ref{HittingTimes}. 
Appendix~\ref{DividedDiff} concerns the notion of divided differences
which are used in Section~\ref{MW}. 
We state three equivalent definitions and prove a related lemma.
In Appendix~\ref{merlo} we write in our context the proof of Theorem~\ref{borgorosa},
which is due to Chang~\cite{C13}
and which we use in Section~\ref{archi}.

%%%%%%%%%%%%%%%%%%%%%%%%%%%%%%%%%%%%%%%%%%%%%%%%%%%%%%%%%%%%%%%%%%%%%%%%%%%%%%%%%%%%%%%%%%%%%%%%%%%%%%%%%%%%%%%%%%%%

To conclude this introductory section and to simplify the reading, we fix here some notation. Further notation will be introduced at the occurrence.

\subsubsection{Main notation}\label{Notation}
\begin{description}
\item[\em Sets] \hfill 
	\begin{description}
	\item[Spaces:] 
		$\mathcal X$ will be our reference state space of size $n$ for the irreducible Markov process $X$.
		In the sequel, we work with extensions of $\mathcal X$ which will be denoted by 
		$\bar{\mathcal X}$ or with more general spaces denoted by $\mathcal Y$.
	\item[Subsets:]
		the symbols $\subset$ and $\subsetneq$ will be used as inclusion and strict inclusion, respectively.
		Subsets will be generally denoted by capital letters: $A,B,R,S$.
	\item[Complement:]
		the complement of a set $A$ will be denoted by $A^c$,
		and it will be clear from the context,
		with respect to which set ${\cal Y} \supset A$ is this complement ${\cal Y} \setminus A$ defined.
	\end{description}

\item[\em Graphs] \hfill
	\begin{description}
	\item[Edges:]
		$\mathcal{E}$ introduced in \eqref{oriented edges} stands for the set of oriented edges on $\mathcal X$.
		For $e = (x, y) \in {\cal E}$, $w(e) = w(x, y)$ is its associated weight.
	\item[Extreme points of oriented edges:]
		for an oriented edge $e=(x,y)\in \mathcal{X}\times \mathcal{X}$,
		we denote the starting and the ending points of $e$
		and by $e_- = x$ and $e^+ = y$. 
	\item[Reversed edges:]
		for an oriented edge $e = (x, y)$, $-e$ stands for its `opposite' $-e = (y, x)$.
	\item[Forest space:]
		$\mathcal{F}$ denotes the set of spanning rooted oriented forests on $\mathcal{X}$.
	\item[A given forest:]
		elements of $\mathcal{F}$ are denoted by $\phi\in \mathcal{F}$ and are identified with subsets of $\mathcal{E}$.
	\item[Roots:]
		given $\phi\in \mathcal{F}$, $\rho(\phi)$ denotes the set of the roots of the trees in $\phi$. 
	\item[Tree associated with a given vertex:]
		given $x$ in ${\cal X}$ and $\phi \in {\cal F}$
		we denote by $\tau_x(\phi)$ the unique maximal tree in $\phi$ that covers $x$.
		We write $\rho(\tau_x(\phi)) = \{x\}$ when $\tau_x(\phi)$ is rooted at $x$.
	\end{description}

\item[\em Matrices] \hfill
	\begin{description}
	\item[Identity:]
		for $A \subset {\cal Y}$ we denote by $\mathbbm{1}_A$
		the identity matrix $\mathbbm{1}_A = (\delta_{xy})_{x, y \in A}$
		and we identify $\mathbbm{1}_A$ with the identity operator
		on the space of functions $f : A \rightarrow \mathbbm{R}$.
		We also write $\mathbbm{1}$ for $\mathbbm{1}_{{\cal Y}}$.
	\item[Restriction of a matrix:]
		given a matrix $M = (M(x, y))_{x, y \in {\cal X}}$, 
		for any subset $A$ of $\mathcal{X}$, $[M]_A$ stands for the restriction of $M$
		to its elements that are doubly indexed in $A$: $[M]_A = (M(x, y))_{x, y \in A}$.
	\item[Determinants:]
		$\det_A(M)$ will denote the determinant of the matrix obtained from $M$ 
		by removing all the lines and columns with indexes outside $A\subset\mathcal{X}$, i.e., $\det_A(M)=\det([M]_A)$.
	\item[Characteristic polynomial:]
		we will often write $\det(\lambda - M)$ instead of $\det(\lambda {\mathbbm 1} - M)$
		for the characteristic polynomial of $M$.
\end{description}

\item[\em Markov processes] \hfill
	\begin{description}
	\item[Continuous time processes and discrete time random walks:]
		$X$ and $Y$ will denote Markov processes on the finite spaces $\mathcal{X}$ and $\mathcal{Y}$, respectively.
		$\hat X$ and $\hat Y$ will denote some associated discrete-time Markov chains,
		which generally {\em are not} obtained from $X$ and $Y$ by taking the sequence of positions of $X$ and $Y$
		at their successive jumps from a position in ${\cal X}$ and ${\cal Y}$ to a distinct one. 
	\item[Generators:]
		$L$ and $\mathcal{L}$ denote the generators of the Markov processes $X$ and $Y$ respectively. 
		For a given subset $A$ of ${\cal Y}$,
		${\cal L}_A$, $[{\cal L}]_A$ and ${\cal L}^A$ respectively
		stand for the Markovian or sub-Markovian generators of 
		the process $Y$ {\em restricted to} $A$,
		the process $Y$ {\em killed outside} $A$ and
		the {\em trace of\/} $Y$ {\em on} $A$ respectively.
		In other words, if the action of ${\cal L}$ on functions $f : {\cal Y} \rightarrow \mathbbm{R}$
		is defined by 
		\begin{equation}
			({\cal L} f)(y) = \sum_{z \in {\cal Y}} \alpha(y, z) [f(z) - f(y)] \,,
			\quad y \in {\cal Y} \,,
		\end{equation}
		then ${\cal L}_A$ is the Markovian generator of a process $Y_A$ on $A$
		and acts on functions $f : A \rightarrow \mathbbm{R}$
		according to 
		\begin{equation}
			({\cal L}_A f)(y) = \sum_{z \in A} \alpha(y, z) [f(z) - f(y)] \,,
			\quad y \in A \,,
		\end{equation}
		$[{\cal L}]_A$ is the sub-Markovian generator obtained from the matrix representation of ${\cal L}$ 
		by deleting the row and columns with indices outside $A$ 
		and acts on functions $f : A \rightarrow \mathbbm{R}$
		according to 
		\begin{equation}
			([{\cal L}]_A f)(y) = -\sum_{z \not \in A} \alpha(y, z) f(y) + \sum_{z \in A} \alpha(y, z) [f(z) - f(y)] \,,
			\quad y \in A \,,
		\end{equation}
		and ${\cal L}^A$ is the Schur complement of $[{\cal L}]_{A^c}$ in ${\cal L}$,
		which defines another Markov process
		$Y^A$ on $A$ (see Appendix~\ref{SCHUR} for more details).
	\end{description}
\end{description}

%%%%%%%%%%%%%%%%%%%%%%%%%%%%%%%%%%%%%%%%%%%%%%%%%%%%%%%%%%%%%%%%%%%%%%%%%%%%%%%%%%%%%%%%%%%%%%%%%%%%%%%%%%%%%%%%%%%%
%NEW SECTION
%%%%%%%%%%%%%%%%%%%%%%%%%%%%%%%%%%%%%%%%%%%%%%%%%%%%%%%%%%%%%%%%%%%%%%%%%%%%%%%%%%%%%%%%%%%%%%%%%%%%%%%%%%%%%%%%%%%%

\section{Background material}\label{Background}

\subsection{On loop-erased trajectories}\label{loops}
We introduce here a slightly more general setting than in Section \ref{setting}.
Let $Y$ be a Markov process on a finite state space $\mathcal{Y}$
with a generator ${\cal L}$ given by
\begin{equation}\label{MarkovExtendedGenerator}
(\mathcal{L}f)(y)=\sum_{z\in \cal Y}\alpha(y,z)[f(z)-f(y)],\quad y\in \mathcal{Y}, 
\end{equation}
with  $f:\mathcal{Y}\rightarrow \mathbb{R}$ arbitrary and $\{\alpha(y,z)\in [0, +\infty] :\: (y,z)\in \mathcal{Y}\times \mathcal{Y}\}$ 
a given collection of non-negative transition rates
(to avoid ambiguities, we assume that for each $y$ in ${\cal Y}$
there is at most one $z$ such that $\alpha(y, z) = + \infty$).
Let
\begin{equation}\label{VariableSpeed2}
\alpha(y)=\sum_{z\in \mathcal{Y}\setminus\{y\}}\alpha(y,z)\in[0,\infty].
\end{equation}
Let $B$ be a subset of $\mathcal{Y}$ such that
\begin{equation}\label{Bcemetery}
 \{y\in \mathcal Y :\: \alpha(y)=+\infty\}\subset B\subseteq \mathcal Y,
\end{equation}
so that, $\alpha(y)<\infty$ for any $y\in B^c =\mathcal{Y}\setminus B$, the complement of $B$ in $\mathcal{Y}$.
We assume that $B$ is accessible from any starting point of the process $Y$.

Denote by $\gamma_B=(y_0,\ldots,y_l)$ a self-avoiding path of $l+1$ points and length $l$ such that $y_i\in B^c$ for $i=0,\ldots,l-1$ and $y_l\in B$.
For $y_0\in B^c$, let ${P}_{y_0}$ the law of the random walk $Y$ when starting from $y_0$. 
Denote by $\Gamma_B$ a random trajectory obtained from $Y$ under ${P}_{y_0}$ as follows:
stop the walk $Y$ when it enters the set $B$ for the first time and erase all its loops. 
After this procedure, we are left with a self-avoiding trajectory of variable length.
In the next proposition we compute the probability that $\Gamma_B$ is the given trajectory $\gamma_B$.
To this end, we use a discrete skeleton of the Markov process $Y$ {\em absorbed in} $B$.
This justify the following definitions.

Set
\begin{equation}\label{ConstantSpeed}
\bar\alpha=\max_{y\in B^c}\alpha(y)<\infty,
\end{equation}
and let $\hat P$ be the $m \times m$ stochastic matrix, with $m=|Y|$, identified by the entries
\begin{equation}\label{rwtransdis}
\hat P(y,z)= \left\{
\begin{array}{cl}
\delta_{y,z} & \text{if } y \in B\\
\alpha(y,z)/\bar{\alpha} & \text{if } z\neq y \text{ and }y\in B^c,\\
1- \displaystyle\sum_{x\in \mathcal{Y}\setminus\{y\}}\hat P(y,x)  & \text{if } y=z \in B^c.
\end{array}
\right.
\end{equation}

Such a matrix $\hat P$ is a Markovian transition matrix for a discrete-time random walk~$\hat Y$ on~$\mathcal Y$.
In particular, for an arbitrary function $f$, by construction we have that
\begin{equation}\label{disVScont}
(\mathcal{L}f)(y)=(\bar{\alpha}(\hat{P}-\mathbbm{1})f)(y), \text{ for all } y\in B^c.
\end{equation}
We are in shape to prove the claimed proposition,
by using a nice independence argument 
we learned from Laurent Tournier.

\begin{proposition}\label{Marchal}
{\bf (Marchal~\cite{M00})}
Consider the random walk $Y$ on $\mathcal{Y}$ with generator $\mathcal{L}$ as in \eqref{MarkovExtendedGenerator}.
Fix distinct points $y_0$, \dots, $y_{l-1}$ in $B^c$ and $y_l \in B$,
so that $\gamma_B = (y_0, \dots, y_l)$ is a self-avoiding path
from $y_0$ to $B$. Then, under ${P}_{y_0}$,
\begin{equation}\label{loop-erasedLaw}
{P}_{y_0}\left(\Gamma_B=\gamma_B\right)=
\prod_{i=0}^{l-1}\alpha(y_{i},y_{i+1}) 
\frac{{\rm det}_{B^c\setminus \{y_0,\ldots,y_{l-1}\}}(-\mathcal{L})}{{\rm det}_{B^c}(-\mathcal{L})} \,,
\end{equation}
with the matrix notation according to Section \ref{Notation}.
\end{proposition}

\begin{proof}For the discrete chain $\hat{Y}$, let 
$\hat{T}^+_{y_{0}}$ and $\hat{T}_B$ be the first return time to $y_0$ and the hitting time of $B$, respectively.
More precisely, $\hat{T}^+_{y_0}=\inf\{k\geq 1:\hat{Y}_k=y_0\}$ and $\hat{T}_B=\inf\{k\geq 0:\hat{Y}_k\in B\}$.
Note that by definition of $\Gamma_B$ we have 
\begin{equation}\label{T1}
{P}_{y_0}\left(\Gamma_B=\gamma_B | \hat{T}^+_{y_0}< \hat{T}_B\right)
={P}_{y_0}\left(\Gamma_B=\gamma_B \right).
\end{equation}
As a consequence, we can write:
\begin{equation}\label{T2}
\begin{aligned}
{P}_{y_0}\left(\Gamma_B=\gamma_B \right)
&= {P}_{y_0}\left(\Gamma_B=\gamma_B | \hat{T}^+_{y_0}< \hat{T}_B\right)
{P}_{y_0}\left(\hat{T}^+_{y_0}< \hat{T}_B\right)\\
&\quad +{P}_{y_0}\left(\Gamma_B=\gamma_B, \hat{T}^+_{y_0}>\hat{T}_B\right)  \\ 
&= {P}_{y_0}\left(\Gamma_B=\gamma_B \right){P}_{y_0}\left(\hat{T}^+_{y_0}< \hat{T}_B\right) \\
&\quad +{P}_{y_0}\left(\Gamma_B=\gamma_B, \hat{T}^+_{y_0}>\hat{T}_B\right).
\end{aligned}
\end{equation}
It follows from \eqref{T2} that
\begin{equation}\label{T3}
{P}_{y_0}\left(\Gamma_B=\gamma_B \right)
=\frac{{P}_{y_0}\left(\Gamma_B=\gamma_B, \hat{T}^+_{y_0}>\hat{T}_B\right)}
{{P}_{y_0}\left(\hat{T}^+_{y_0}>\hat{T}_B\right)}.
\end{equation}

Denote by $\ell_{y_0}(\hat{T}_B)$ the local time $\hat{Y}$ spends at $y_0$ before entering $B$, and by 
$[\hat{P}]_{B^c}$ the matrix $\hat{P}$ restricted to $B^c$, then
\begin{equation}
\label{T4}
\begin{aligned}
\frac{1}{{P}_{y_0}\left(\hat{T}^+_{y_0}> \hat{T}_B\right)}&=
{E}_{y_0}[\ell_{y_0}(\hat{T}_B)]=\sum_{k\geq 0}[\hat{P}]^k_{B^c}(y_0,y_0)\\ &= ([\mathbbm{1}-\hat{P}]_{B^c})^{-1}(y_0,y_0)=
\frac{{\rm det}_{B^c\setminus \{y_{0}\}}(\mathbbm{1}-\hat{P})}{{\rm det}_{B^c}(\mathbbm{1}-\hat{P})},
\end{aligned}
\end{equation}
where the last equality follows by Cramer's formula for an inverse matrix.

On the other hand, for the numerator in the r.h.s. of equation
\eqref{T3}, we can write
\begin{equation}
\label{T5}
\begin{aligned}
& {P}_{y_0}\left(\Gamma_B=(y_0,\ldots,y_l), \hat{T}^+_{y_0}>\hat{T}_B\right)
  =\hat{P}(y_0,y_1){P}_{y_1}\left(\Gamma_{B\cup\{y_0\}}=(y_1,\ldots,y_l)\right).
\end{aligned}
\end{equation}

By plugging \eqref{T4} and \eqref{T5} into \eqref{T3}, and iterating, we have that  
\begin{equation}
\begin{aligned}
{P}_{y_0}\left(\Gamma_B=\gamma_B \right) 
= &\hat{P}(y_0,y_1){P}_{y_1}\left(\Gamma_{B\cup\{y_0\}}=(y_1,\ldots,y_l)\right)
\frac{{\rm det}_{B^c\setminus \{y_{0}\}}(\mathbbm{1}-\hat{P})}{{\rm det}_{B^c}(\mathbbm{1}-\hat{P})} \\
=& \cdots=\prod_{i=0}^{l-1}\hat {P}(y_{i},y_{i+1})\frac{{\rm det}_{B^c\setminus \{y_0,\ldots,y_{l-1}\}}(\mathbbm{1}-\hat{P})}{{\rm det}_{B^c}(\mathbbm{1}-\hat{P})}\\
=&\bar{\alpha}^{-l}\prod_{i=0}^{l-1}\alpha(y_{i},y_{i+1})\frac{{\rm det}_{B^c\setminus \{y_0,\ldots,y_{l-1}\}}(-\mathcal{L}/\bar{\alpha})}{{\rm det}_{B^c}(-\mathcal{L}/\bar{\alpha})}\\
=&\prod_{i=0}^{l-1}\alpha(y_{i},y_{i+1}) 
\frac{{\rm det}_{B^c\setminus \{y_0,\ldots,y_{l-1}\}}(-\mathcal{L})}{{\rm det}_{B^c}(-\mathcal{L})}.
\end{aligned}
\end{equation}
\end{proof}

\subsection{Wilson's algorithm}
We introduce here the algorithm due to Wilson and Propp~\cite{PW98} which allow us 
to sample the measure \eqref{ForestProb}.
First, we extend the Markov process $X$ defined through \eqref{MarkovGenerator} on $\mathcal X$ 
to a Markov process $\bar X$ on the space $\bar{\mathcal X}=\mathcal X \cup \{\Delta\}$ 
by interpreting $\Delta$ as an absorbing state and by adding some killing rates.
Consider the space $\mathcal{X}$, with $|\mathcal{X}|=n$.
Assume a collection of killing rates $\{q(x)\in[0,\infty] :\: x\in\mathcal{X}\}$ is given, and let $Q$ be the diagonal matrix
with diagonal entries $q(x)$, $x \in {\cal X}$.
Consider the Markov process $\bar X$ on the finite state space $\bar{\mathcal X}$ with generator given by
\begin{equation}\label{ExtendedMarkovGenerator}
(\mathcal Lf)(x)=\left\{
\begin{array}{cl}
(Lf)(x)+ q(x)[f(\Delta)-f(x)],&\mbox{ if } x\in \mathcal{X},\\ 
0,&\mbox{ if } x=\Delta.
\end{array}
\right.
\end{equation}
with  $f:\bar{\mathcal{X}}\rightarrow \mathbb{R}$ arbitrary, and $L$ defined in \eqref{MarkovGenerator}.
In particular, the matrix $-\mathcal{L}$ associated with the generator in \eqref{ExtendedMarkovGenerator} satisfies
\begin{equation}[-\mathcal{L}]_\mathcal{X}=Q-L.\end{equation}

Next, we describe the algorithm.
For any $A\subset \bar{\mathcal X}$, note that, due to irreducibility, $T_A=\inf\{t\geq 0: \bar{X}_t \in A\}$ is a.s. finite. 
\begin{definition}{\bf{(Wilson's algorithm)}}
\begin{enumerate}\label{Wilson}
\item Start the process $\bar{X}$ from any point $x_1\in\mathcal X$ until it reaches the absorbing state $\Delta$.
\item Erase all the loops of the trajectory described by $\bar{X}$ up to time $T_\Delta$. Call $\gamma_1(\Delta)$ this self-avoiding trajectory.
($\gamma_1(\Delta)$ is such that $\Delta$ is the last point in $\gamma_1(\Delta)$.)
\item If $\gamma_1(\Delta)$ covers the whole $\bar{\mathcal{X}}$ stop, 
else pick any point $x_2\in\mathcal{X}\setminus V(\gamma_1(\Delta))$,
with $V(\gamma_1(\Delta))$ denoting the set of points covered by $\gamma_1(\Delta)$. 
Start the process $\bar{X}$ from $x_2$ until it hits the set $V(\gamma_1(\Delta))$.
\item Erase all the loops of the trajectory described by $\bar{X}$ starting from $x_2$ up to time $T_{V(\gamma_1(\Delta))}$. 
Call $\gamma_2(\Delta)$ this self-avoiding trajectory.
\item If $\displaystyle\cup_{i=1,2}V(\gamma_i(\Delta))=\bar{\mathcal{X}}$ stop, 
else pick any point $x_3\in\mathcal X\setminus \displaystyle\cup_{i=1,2}V(\gamma_i(\Delta))$. 
Start the process $\bar{X}$ from $x_2$ until it hits the set $\displaystyle\cup_{i=1,2}V(\gamma_i(\Delta))$.
\item Iterate until $\bar{\mathcal{X}}$ is covered.
\end{enumerate}
\end{definition}

Denote by $\mathcal{T}_{\bar{\mathcal X}}$ the set of spanning oriented trees on $\bar{\mathcal X}$ rooted at $\Delta$.
This algorithm produces in finite time an element $\tau$ of $\mathcal{T}_{\bar{\mathcal X}}$.
As a corollary of Proposition \ref{Marchal}, we can easily compute the probability that the algorithm produces a given $\tau\in\mathcal{T}_{\bar{\mathcal X}}$.
\begin{corollary}\label{WilsonProb}
Fix a tree $\tau\in\mathcal{T}_{\bar{\mathcal X}}$.
Let $\partial{\rho(\tau)}=\{e \in \tau :\: e^+ = \Delta\}$
be the set of edges in $\tau$ that point to the root $\Delta$.
Denote by $\mathbb{P}({\cal W}=\tau)$
the probability that Wilson's algorithm produces the tree $\tau$.
Then 
\begin{equation}\label{WilProb}
\mathbb{P}({\cal W}=\tau)
=\frac{\left[
	\displaystyle\prod_{e\in\partial{\rho(\tau)} :\: e_- \not \in  S} q(e_-)
\right] \left[
	\displaystyle\prod_{e\in\tau\setminus\partial{\rho(\tau)}}w(e)
\right]} {{\rm det}_{{\cal X} \setminus S}(-\mathcal{L})}
\mathbbm{1}_{\{S \subset \partial_- \rho(\tau)\}} \,,
\end{equation}
with $\partial_- \rho(\tau)  = \cup_{e \in \partial \rho(\tau)} \{e_-\}$.
\end{corollary}

\begin{proof}
Recall the notation in Proposition \ref{Marchal}. Set $\mathcal Y=\bar{\mathcal X}=\mathcal X \cup \{\Delta\}$ and $Y=\bar{X}$.
Start with $B=\{\Delta\}\cup S$. By the definition of the algorithm, the proof follows by iterating the formula in equation \eqref{loop-erasedLaw} 
where at each iteration we set the right
$B$ according to the given tree $\tau$.
Whatever the choice of starting points in Wilson's algorithm we get the same result.
\end{proof}

We conclude this section by observing that there exists a natural bijection between $\mathcal F$ and $\mathcal{T}_{\bar{\mathcal X}}$.
Indeed, given $\phi\in\mathcal F$, let $\tau(\phi)$ be the unique element in $\mathcal{T}_{\bar{\mathcal X}}$ 
obtained from $\phi$ by adding all the edges connecting the roots in $\phi$ to $\Delta$, 
i.e. add all edges $e$ such that $e_-\in\rho(\phi)$ and $e^+=\Delta$.
Vice versa, given $\tau\in\mathcal{T}_{\bar{\mathcal X}}$, by removing all edges $e\in\partial{\rho(\tau)}$ 
we can identify a unique element $\phi\in\mathcal F$.
This simple observation together with Corollary \ref{WilsonProb} 
allow us to sample the measure in \eqref{ForestProb} using Wilson's algorithm. 

\begin{figure}[hbtp]
	\caption{
		\label{mimosa} Samples from $\nu_q$ for $q = .001$ on the two-dimensional $512 \times 512$ torus 
		with uniform rates equal to 1 between nearest neighbours for the first picture and, for the second picture, 
		with an additional northward drift, such that $w(x, y) = 1.2$ if $y$ is the northern
		nearest neighbour of $x$ and $w(x, y) = 1$ if $y$ is its eastern, western of southern nearest neighbour. 
		The third picture is a sample from $\nu_q$ for the same $q = .001$
		on the $987 \times 610$ rectangular grid and for Metropolis random walk in a Browian sheet potential $V$,
		i.e., nearest neighbour rates are given by $w(x, y) = \exp\{-\beta[V(y) - V(x)]_+\}$ with $\beta = .04$
		and $V$ is the restriction to the grid of a Brownian sheet with 0 value on the west and north sides of
		the box. In each picture different blue levels are given to points in different trees, cyan lines separate
		neighbouring trees, and the forest roots are at the centers of red diamonds.
	}
	\begin{center}
	\includegraphics[width=200pt]{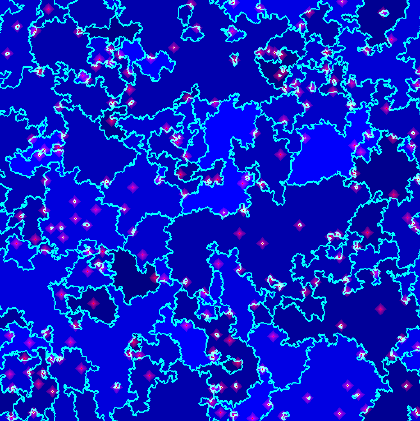}\hspace{18pt}\includegraphics[width=200pt]{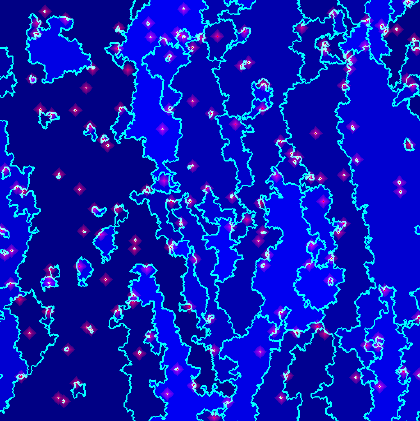}

	\vspace{18pt}
	\includegraphics[width=420pt]{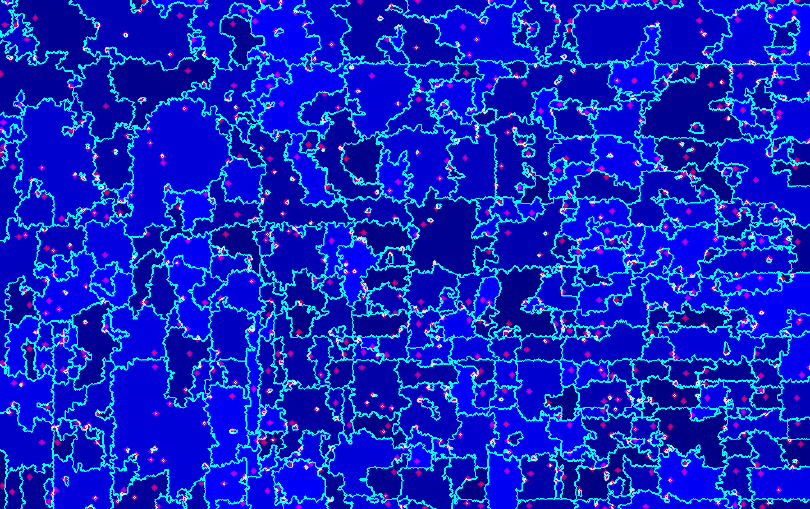}
	\end{center}
\end{figure}

%%%%%%%%%%%%%%%%%%%%%%%%%%%%%%%%%%%%%%%%%%%%%%%%%%%%%%%%%%%%%%%%%%%%%%%%%%%%%%%%%%%%%%%%%%%%%%%%%%%%%%%%%%%%%%%%%%%%%%%%%%%%%%%%%%
%%%%%%%%%%%%%%%%%%%%%%%%%%%%%%%%%%%%%%%%%%%%%%%%%%%%%%%%%%%%%%%%%%%%%%%%%%%%%%%%%%%%%%%%%%%%%%%%%%%%%%%%%%%%%%%%%%%%%%%%%%%%%%%%%%%
\section{The root process}\label{ForRoots}

\subsection{Partition function and root number distribution}

We start here to analyze the measure introduced in \eqref{ForestProb} on the space $\mathcal{F}$ of spanning rooted oriented forests on $\mathcal X$.
We compute the partition function, we identify the distribution of the number of roots in the standard case
and we prove that the root process is a determinantal one.
\begin{theorem}\label{PartitionThm} 
{\bf (Partition function and spectrum)}
Assume a collection of killing rates $\{q(x)\in[0,\infty] :\: x\in\mathcal{X}\}$ is given, and let $Q$ be the diagonal matrix 
with diagonal entries $q(x)$, $x \in {\cal X}$.
Let $\nu_Q$ the probability measure on $\mathcal F$ in \eqref{ForestProb}.
Then \begin{equation}\label{PartitionFormula}
 Z_Q={\rm det}_{{\cal X} \setminus S}(Q-L).
\end{equation}
and, recalling the notation from Section \ref{Wilson},
\begin{equation}\label{samplingPhi}
	\nu_Q(\phi)=\mathbb{P}({\cal W}=\tau(\phi)).
\end{equation}
In the case $q(x)\equiv q>0$, we recover the standard probability measure in \eqref{StandardForestProb}, 
and the standard partition function in \eqref{StandardPartitionFn}
is given by the characteristic polynomial of $L$
\begin{equation}\label{StandardPartitionFormula}
 Z(q)=\det(q - L)=\prod_{i=0}^{n-1}(q+\lambda_i)=q\prod_{i=1}^{n-1}(q+\lambda_i)
\,,
\end{equation}
where the $\lambda_i$'s are the eigenvalues of $-L$ ordered by non-decreasing real part.
When $R \neq \emptyset$ and $q \geq 0$, 
the partition function is given by the characteristic polynomial
of the sub-Markovian generator of the process killed in $R$:
\begin{equation} \label{lucien}
	Z_R(q) = \det\left(q - [L]_{R^c}\right)
	\,.
\end{equation}
\end{theorem}

\par\noindent {\bf Remark:}
This kind of results goes back to Kirchhoff~\cite{K847}.
Here, like in~\cite{CA02}, Theorem 3', we include the non-reversible case and stress the dependence
in $q$.

\begin{proof}
As observed in the previous section, for each forest $\phi\in \mathcal{F}$ there is a unique $\tau(\phi)\in \mathcal{T}_{\bar{\mathcal X}}$. 
By Corollary \ref{WilsonProb} we then have that
\begin{equation}
\mathbb{P}({\cal W}=\tau(\phi))=\frac{\left[\displaystyle\prod_{e\in\partial{\rho(\tau)} :\: e_- \not \in S }q(e_-)\right]
\left[\displaystyle\prod_{e\in\tau\setminus\partial{\rho(\tau)}}w(e)\right]}
{{\rm det}_{{\mathcal X}\setminus S}(-\mathcal{L})} 
\mathbbm{1}_{\{S \subset \partial_- \rho(\tau)\}}
= \frac{w_Q(\phi)}{{\rm det}_{{\cal X} \setminus S}(Q-L)} \,.
\end{equation}
By summing over all $\phi\in\mathcal F$ we immediately get \eqref{PartitionFormula}. In fact,
\begin{equation}
	Z_Q=\sum_{\phi\in\mathcal F}w_Q(\phi)={\rm det}_{{\cal X} \setminus {\cal S}}(Q-L)\sum_{\phi\in\mathcal F}\mathbb{P}({\cal W}=\tau(\phi))
	={\rm det}_{{\cal X} \setminus {\cal S}}(Q-L)
	\,.
\end{equation}
Moreover, if $q(x)\equiv q>0$, then $Q=q\mathbbm{1}$ and $Z(q)={\rm det}[q\mathbbm{1}-L]=\chi_L(q)$.
\end{proof}

In the standard case, an immediate consequence of Theorem \ref{PartitionThm} is a characterization 
of the law of the cardinality of the set of roots.
\begin{corollary}\label{Bernoullicity}
{\bf (Root number distribution)}
Assume the standard case $q(x)\equiv q>0$ and that $L$ has real spectrum. 
Let $N$ be a sum of $n$ independent Bernoulli random variables with parameters $q/(q+\lambda_i)$. 
The random variable $|\rho(\Phi_q)|$ counting 
the number of roots (or equivalently, of trees) in $\Phi_q$ has the same law as $N$. 
\end{corollary}

\begin{proof}
Observe that the coefficient of degree $k$ in 
\begin{equation}
Z(q) = \sum_\phi q^{|\rho(\phi)|} w(\phi)
\end{equation}
is the total weight of the set of forest with exactly $k$ roots.
Since $Z(q) = \prod_i (q + \lambda_i)$ we get,
\begin{equation}
\begin{aligned}
{\mathbb P}\left(|\rho(\Phi_q)|= k \right)
= & \frac{1}{Z(q)}\sum_{\phi\in \mathcal{F} : |\rho(\phi)|= k} w_q(\phi)
= \left(\displaystyle\prod_{i=0}^{n-1}(q+\lambda_i)\right)^{-1} q^k
\sum_{I\in \mathcal{P}[n-k]}\prod_{i\in I}\lambda_i\\
&= \sum_{I\in \mathcal{P}[n-k]}\left[\prod_{i\in I}\left(\frac{\lambda_i}{q+\lambda_i}\right)\right]
\left[\prod_{j\notin I}\left(\frac{q}{q+\lambda_j}\right)\right]\\
&= \sum_{J\in \mathcal{P}[k]}
\left[\prod_{j\in J}\left(\frac{q}{q+\lambda_j}\right)\right]
\left[\prod_{i\notin J}\left(1-\frac{q}{q+\lambda_i}\right)\right],
\end{aligned}
\end{equation}
where $\mathcal{P}[k]$ stands for the set of all possible $k$ elements of the set $\{0,1,\ldots,n-1\}$.
\end{proof}

\par\noindent
{\bf Remark:} When the spectrum of $L$ does contain a non-real part,
one can still compute the law of $|\rho(\Phi)|$ and get the same algebraic
expressions in terms of the eigenvalues.
One can also compute momenta by differentiating with respect to $q$
the logarithm of the partition function. In particular, the mean value and the variance are given by
\begin{align}
&{\mathbb E}_q[\rho(\Phi)] = \sum_{i = 0}^{n - 1} \frac{q}{q + \lambda_i},	\label{swizzera}\\
&{\mathbb V}\!{\rm ar}_q(\rho(\Phi)) = \sum_{i = 0}^{n - 1} \frac{q}{q + \lambda_i} - \left(\frac{q}{q + \lambda_i}\right)^2.
\end{align}
We note however that the contribution of the imaginary part of the eigenvalues can make uneasy the comparison 
between variance and mean value, at least for small values of $q$.
This is the reason why, when dealing with the question of getting samples with a number of roots that approximates
a given $m \leq n$, we will restrict ourselves to the real spectrum case.

\subsection{Determinantal structure}

Next, we prove that the random set $\rho(\Phi_q)$, or more generally $\rho(\Phi_Q)$, is a determinantal process
as suggested after~\cite{HKPV06} by the previous result.
This is the content of Theorem \ref{DeterminantalRoots} for which we will present an algebraic and a probabilistic proof.

Let us first show a simple lemma.
Consider the Markov process $\bar X$ on $\bar {\mathcal {X}}=\mathcal X \cup \{\Delta\}$ defined via its generator in \eqref{ExtendedMarkovGenerator}.
This process can be coupled up to time $T_\Delta$
with a Markov process $X$ in ${\cal X}$
which is stopped at rate $q(x)$ in any $x \in {\cal X}$.
Calling $T_Q$ this stopping time,
$X(T_Q)$ is then the last point visited in ${\cal X}$
by the process $\bar{X}$ before time $T_{\Delta}$,
or, in other words, the end point of the last egde $e \in {\cal E}$ crossed inside ${\cal X}$, before going
directly to $\Delta$.

Let $K_Q$ be the Markovian matrix defined by 
\begin{equation}\label{ExtKernel}
	K_Q(x,y)={P}_x(X(T_Q) = y),  \quad \text{for } x,y\in\mathcal X
	\,.
\end{equation}
This transition kernel can also be expressed in terms of the Green's function $G_Q$:
\begin{lemma}\label{GreenKernel}
	For all $x$ and $y$ in ${\cal X}$
	\begin{equation} \label{ExtKernelGreen}
		K_Q(x,y) = \left\{
			\begin{array}{ll}
				G_{Q}(x,y)q(y)
				& \mbox{if $x, y \not \in S$,} \\
				P_x(X(T_Q) = X(T_S) = y)
				& \mbox{if $y \in S$,} \\
				{\mathbbm 1}_{\{x = y\}}
				& \mbox{if $x \in S$}
			\end{array}
		\right.
	\end{equation}
	with $S$ defined in~\eqref{gammino} and, for $x, y \not\in {\cal X}$,
	\begin{equation} \label{ExtendedGreenFn}
		G_{Q}(x,y)
		={E}_x\left[
			\ell_y(T_{Q})
		\right]
		= [Q-L]_{{\cal X} \setminus S}^{-1}(x,y)
	\end{equation}
	where $\ell_y(T_Q)$ is the local time up to $T_Q$.
\end{lemma}
\begin{proof}
If $x$ or $y$ belong to $S$,
then~\eqref{ExtKernelGreen} is just a restatement
of the definition of $K_Q$.
We then assume $x, y \not \in S$
and use the notation of Section \ref{Marchal} to work in discrete time.
Set $\mathcal Y=\bar{\mathcal X}$, $Y=\bar X$, $B=\{\Delta\} \cup S$ and write
\begin{equation}
\begin{aligned}
K_Q(x,y)&={P}_x({X}(T_{Q})=y)=\sum_{k\geq 1}{P}_x(\hat{Y}(k-1)=y,\hat{T}_{\Delta}=k)\\
&=\sum_{k\geq 1}{P}_x(\hat{Y}(k-1)=y,\hat{T}_{\Delta} > k-1)\hat{P}(y,\Delta)= {E}_x[\ell_y(\hat{T}_{\Delta})]\frac{q(y)}{\bar{\alpha}}
=G_{Q}(x,y)q(y).
\end{aligned}
\end{equation}
\end{proof}

Note that, 
when $R \neq \emptyset$ or $q > 0$,
we have
\begin{equation}\label{cappello}
K_{q, R}(x,y)={P}_x(X(T_q \wedge T_R)=y),  \quad \text{for } x,y\in\mathcal X,
\end{equation}
with $T_q$ being an independent exponential random variable of parameter $q$.
In particular, if $G_q(x,y)={E}_x[\ell_y(T_q)]$ is the Green's function up to time $T_q$,
then \begin{equation}
\label{KernelGreen}
K_q(x,y)=qG_q(x,y)=q(q\mathbbm{1}-L)^{-1}(x,y).
\end{equation}

\begin{theorem}\label{DeterminantalRoots}{\bf (Determinantal roots)}
The root process $\rho(\Phi_Q)$ is a determinantal process
with kernel $K_Q$.
Equivalently, for any $A\subset \mathcal{X}$:
\begin{equation}\label{nastro}
	{\mathbb P}\left(A\subset \rho\left(\Phi_Q\right)\right)={\rm det}_{A}(K_Q)
	\,.
\end{equation}
\end{theorem}

\begin{figure}[hbtp]
	\caption{
		\label{caffe`}
		On the two dimensional torus,
		the difference between the law of our root set
		and that of a product of Bernoulli measure is
		far too subtle to be detected on a single sample.
		Walking away from translation invariant models,
		we can find much more and huge correlations.
		Here are pictures of a sample on the two-dimensional
		torus with uniform rates equal to 1 and $q = .002$
		and of a sample associated with 
		the Metropolis random walk on the square grid
		in a Brownian sheet potential with inverse temperature
		$\beta = .16$ and extinction rate $q = 10^{-4}$. 
	}
	\begin{center}
		\includegraphics[width=200pt]{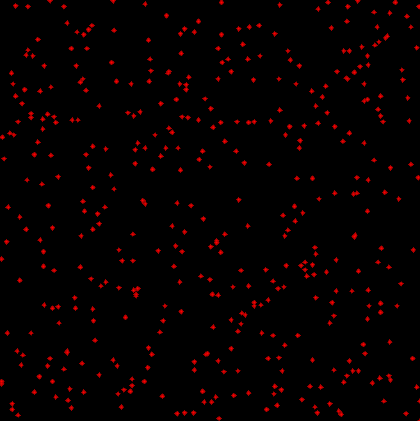}
		\hspace{18pt}
		\includegraphics[width=200pt]{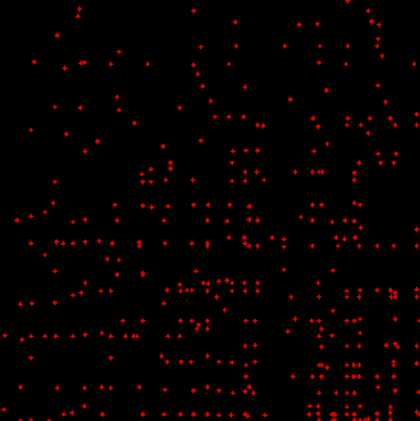}
	\end{center}
\end{figure}

\begin{proof}{\bf{(Algebraic proof of Thm \ref{DeterminantalRoots})}}
Assume first $S=\emptyset$. Consider a set $A\subset \mathcal{X}$ with $|A|=r$ of the form $A=\{x_1,x_2,\ldots,x_r\}$.
By choosing the different points in $A$ as starting point at each iteration in Wilson's algorithm
(remember that the law of the obtained tree does not depend
on the order of the starting points), 
by \eqref{loop-erasedLaw}, we get
\begin{equation}\label{A1}
\begin{aligned}
{\mathbb P}\left(\{x_1,x_2,\cdots,x_r\}\subset \rho\left(\Phi_Q\right)\right)&=
q(x_1)\frac{{\rm det}_{\{x_1\}^c}(Q-L)}{{\rm det}_{\mathcal{X}}(Q-L)}
q(x_2)\frac{{\rm det}_{\{x_1,x_2\}^c}(Q-L)}{{\rm det}_{\{x_1\}^c}(Q-L)}
\times\cdots \\
&\cdots\times q(x_r)\frac{{\rm det}_{A^c}(Q-L)}{{\rm det}_{\{x_1,\ldots,x_{r-1}\}^c}(Q-L)}=
\left[\prod_{i=1}^rq(x_i)\right]\frac{{\rm det}_{A^c}(Q-L)}{{\rm det}_{\mathcal{X}}(Q-L)}\\
&=\frac{{\rm det}_{A^c}(Q-L)}{{\rm det}_{\mathcal{X}}(Q-L)}\displaystyle{\rm det}_{A}(Q).
\end{aligned}
\end{equation}
In case $A=\mathcal X$, the claim is straightforward since equation \eqref{A1} reads
\begin{equation}
{\mathbb P}\left(\rho\left(\Phi_Q\right)=\mathcal{X}\right)= \frac{{\rm det}(Q)}{{\rm det}(Q-L)}
={\rm det}\left((Q-L)^{-1}\right){\rm det}(Q)={\rm det}\left((Q-L)^{-1}Q\right),
\end{equation}
and the r.h.s. equals ${\rm det}(K_Q)$ due to Lemma \ref{GreenKernel}.

In case $A\subsetneq\mathcal X$, we can use the Schur complement (see \eqref{detM-1} in Appendix~\ref{Schur}) to show that
\begin{equation}
\frac{{\rm det}_{A^c}(Q-L)}{{\rm det}_{\mathcal{X}}(Q-L)}={\rm det}_{A}\left((Q-L)^{-1}\right).
\end{equation}
Therefore, from equation \eqref{A1}, we have that
\begin{equation}
	{\mathbb P}_Q\left(\{x_1,x_2,\cdots,x_r\}\subset \rho\left(\Phi\right)\right)=
	{\rm det}_{A}((Q-L)^{-1}){\rm det}_{A}(Q)={\rm det}_{A}((Q-L)^{-1}Q)={\rm det}_{A}(K_Q)
	\,,
\end{equation}
where the second to the last equality is justified because $Q$ is a diagonal matrix and the last equality follows from \eqref{ExtKernelGreen}.

When $S \neq \emptyset$, the proof is the same.
We just have to subtract $S$ from the various considered sets
to get
\begin{equation}
	{\mathbb P}\left(
		A \subset \rho(\Phi_Q)
	\right) = {\rm det}_{A \setminus S}\left(
		K_Q
	\right)
	\,,
\end{equation}
which is equivalent to~\eqref{nastro}
since $K_Q$ is a bloc-diagonal matrix
(recall \eqref{ExtKernelGreen}).
\end{proof}

\begin{proof}{\bf{(Probabilistic proof of Thm \ref{DeterminantalRoots})}}
To avoid a heavy notation, we consider only the case $S=\emptyset$.
Starting from the Markov process $X$ and the killing rates $q(x)$, 
we construct two different absorbing states $\Delta_A$ and $\Delta_{A^c}$ accessible from the set $A$ and $A^c$, respectively.
Set $\mathcal Y=\mathcal{X}\cup \{\Delta_A,\Delta_{A^c}\}$.
Let $Y$ be the Markov process with generator 
\begin{equation}\label{2cemeteriesGen}
\begin{aligned}(\mathcal Lf)(x)&=
(Lf)(x)+ q(x)\mathbbm{1}_{\{x\in A\}}[f(\Delta_A)-f(x)]\\
&+q(x)\mathbbm{1}_{\{x\in A^c\}}[f(\Delta_{A^c})-f(x)], \quad \text{ if } x\in \mathcal{X}, 
\end{aligned}
\end{equation}
and $(\mathcal Lf)(x)=0$ if $x\in\{\Delta_A,\Delta_{A^c}\}$ 
with  $f:\mathcal{Y}\rightarrow \mathbb{R}$ and $L$ as in \eqref{MarkovGenerator}.

Next, consider the subspace $\bar{\bar{A}}=A\cup \{\Delta_A,\Delta_{A^c}\}\subset \mathcal{Y}$.
Let $Y^{\bar{\bar{A}}}$ be the Markov process with state space $\bar{\bar{A}}$ obtained as the trace of the process $Y$ on $\bar{\bar{A}}$.
Let us remark two features of Wilson's algorithm.
First, Wilson's algorithm can be extended to the case of a state space with
more than one absorbing state. 
In this case it produces a rooted spanning forest instead of a tree.
Second, Wilson's algorithm is uniquely determined once 
we fix a state space with some absorbing set
and a Markov generator.
These observations justify the following definitions.
Let $D$ be the set of ending points of the edges starting from $A$ after running Wilson
algorithm on $\bar{\cal X}$ with absorbing state $\Delta$ and generator ${\cal L}$.
Similarly let $D'$ associated in the same way with $A$, when Wilson's algorithm
is run on $\bar{\bar A}$ with absorbing set $\{\Delta_A, \Delta_{A^c}\}$
and generator ${\cal L}^{\bar{\bar A}}$.
Observe at this point that
\begin{equation}
	{\mathbb P}(A\subset \rho(\Phi_Q)) = {\mathbb P}(D = \{\Delta\}) = {\mathbb P}(D' = \{\Delta_A\}),
\end{equation}
and, by using Proposition \ref{Marchal}, compute
\begin{equation}\label{A2}\begin{aligned}
{\mathbb P}\left(A\subset \rho\left(\Phi_Q\right)\right)&
= {\mathbb P}(D' = \{\Delta_A\})
=\frac{
	\prod_{a \in A} q(a)
}{
	{\rm det}_{\bar{\bar{A}}\setminus \{\Delta_A,\Delta_{A^c}\}}(-\mathcal{L}^{\bar{\bar{A}}})
} \\
&= \frac{
	{\rm det}_{A}(Q)
}{
	{\rm det}_{A}(-\mathcal{L}^{\bar{\bar{A}}})
}
={\rm det}\left(\left(\left[-\mathcal{L}^{\bar{\bar{A}}}\right]_A\right)^{-1}\right){\rm det}_{A}(Q)\\
&=
{\rm det}_A\left(G^{\bar{\bar{A}}}_{\{\Delta_A,\Delta_{A^c}\}}\right){\rm det}_{A}(Q),
\end{aligned}
\end{equation}
where $G^{\bar{\bar{A}}}_{\{\Delta_A,\Delta_{A^c}\}}$ denotes the Green's function of the process 
$Y^{\bar{\bar{A}}}$ stopped when entering the absorbing states $\{\Delta_A,\Delta_{A^c}\}$.
Note now that for $x,y\in A$,
\begin{equation}\label{A3}
G^{\bar{\bar{A}}}_{\{\Delta_A,\Delta_{A^c}\}}(x,y)=G^{\mathcal{Y}}_{\{\Delta_A,\Delta_{A^c}\}}(x,y)=G^{\bar{\mathcal{X}}}_{\{\Delta\}}(x,y)
\end{equation}
with $\bar{\mathcal X}=\mathcal{X}\cup\{\Delta\}$ 
and $G^{\bar{\mathcal{X}}}_{\{\Delta\}}$ being the Green's function of the process $\bar{X}$ in \eqref{ExtendedMarkovGenerator}.
Finally, since $G^{\bar{\mathcal{X}}}_{\{\Delta\}}(x,y)=(Q-L)^{-1}(x,y)$ for $x,y\in\mathcal X$, the claim follows by combining equations
\eqref{A2} and \eqref{A3}.
\end{proof}

\subsection{Cumulants}\label{Cumulants}

In this section we compute the cumulants of the determinantal process $\rho(\Phi_q)$:
they are given by a nice formula it is worth to notice.
Let us associate with our random forests $\Phi_q$
with law $\nu_q$, the random variables
\begin{equation}
	\eta_x = {\mathbbm 1}_{\{x \in \rho(\Phi_q)\}}, \quad x \in {\cal X},
\end{equation}
note that they completely describe the root process.
For $A = \{x_1, \dots, x_k\} \subset {\cal X}$  with distinct $x_i$'s, 
the cumulants of these random variables are defined by
\begin{equation}
	\kappa_A(\eta) 
	=\kappa(\eta_{x_1}, \dots, \eta_{x_k})
	= \frac{\partial^k}{\partial \lambda_1 \dots \partial \lambda_k}
	\ln {\mathbb E}\left[
		\exp\left\{
			\sum_{i = 1}^k \lambda_i \eta_{x_i}
		\right\}
	\right] \Bigg|_{\lambda = 0}
	.
\end{equation}		
These quantities are the so-called truncated correlation functions,
that can also be recursively defined by
\begin{equation} \label{luce}
	{\mathbb E}\left[
		\prod_{x \in A} \eta_x 
	\right]
	= \sum_{\Pi \in {\cal P}_A} \prod_{B \in \Pi} \kappa_B(\eta),
\end{equation}
where ${\cal P}_A$ stands for the set of partitions of $A$.

The determinantal nature of the root process makes
its cumulants easy to compute.
With $A \subset {\cal X}$ and ${\cal S}_A$ being the permutation group on $A$,
one has
\begin{equation}
	{\mathbb E}\left[
		\prod_{x \in A} \eta_x
	\right]
	= {\mathbb P}\left(
		A \subset \rho(\Phi_q)
	\right) = {\rm det}_A(K_q) 
	= \sum_{\sigma \in {\cal S}_A} (-1)^{{\rm sgn}(\sigma)} \prod_{x \in A} K_q(x, \sigma(x)).
\end{equation}
Making a cycle decomposition of each permutation in this sum and denoting by ${\cal C}_B$ the set of long cycles on $B \subset A$,
i.e. the set of cycles of length $|B|$ in $B$, after some simple algebra, we get
\begin{equation}
	{\mathbb E}\left[
		\prod_{x \in A} \eta_x
	\right]
	= \sum_{\Pi \in {\cal P}_A} \prod_{B \in \Pi} \sum_{\sigma \in {\cal C}_B} (-1)^{|B| - 1} \prod_{x \in B} K_q(x, \sigma(x)).
\end{equation}
This identifies our cumulants through (\ref{luce}) and gives the following lemma.
\begin{lemma}\label{nebbioline}
For all $A \subset {\cal X}$
\begin{equation}
	\kappa_A(\eta) = (-1)^{|A| - 1} \sum_{\sigma \in {\cal C}_A} \prod_{x \in A} P_x(X(T_q) = \sigma(x)),
\end{equation}
where ${\cal C}_A$ stands for the set of cycles of length $|A|$ in $A$.
\end{lemma}

\par\noindent
{\bf Remark:} 
In the case of uniformly equal weights between nearest neighbours, for large $q$,
$(-1)^{|A| - 1} \kappa_A(\eta)$ behaves like 
the natural low temperature partition function associated with an embedded travelling salesman problem.
In this regime, on the one hand, Wilson's algorithm quickly provides perfect samples
of the root process and, on the other hand, the cumulant is the expected value of some observable for the system 
made of $n$ independent copies of $\rho(\Phi_q)$~\cite{S75}. This suggests that one could find
a practical way to estimate this low temperature partition function and then solve the travelling salesman
problem. Unfortunately, the corresponding observable has an exponentially small probability to
be different from $0$ and consequently, it is in reality impossible to estimate its mean in this way.

%%%%%%%%%%%%%%%%%%%%%%%%%%%%%%%%%%%%%%%%%%%%%%%%%%%%%%%%%%%%%%%%%%%%%%%%%%%%%%%%%%%%%%%%%%%%%%%%%%%%%%%%%%%%%%%%%%%%%%%%%%%%%%%%%%%%%%%%%%%%%%
\section{Hitting times}\label{HittingTimes}

In this section we answer the question raised in the introductory Section \ref{chessboard}.
To this end, we focus on hitting times of a given subset $R\subset \mathcal{X}$, i.e., 
\begin{equation}
	T_R = \inf\{t \geq 0: X(t) \in R\}
	\,.
\end{equation}
We will also look at the return time
\begin{equation} \label{stella}
	T_R^{++} = \inf\left\{
		t \geq 0 :\:
		\exists s > 0,
		X(s) \neq X(0),
		s \leq t,
		X(t) \in R
	\right\}
	\,.
\end{equation}
The reason why we use this heavy double~$^+$ notation
is that we will also consider the often more useful
{\em randomized} or {\em skeleton return time} $T_R^+$,
which is defined as follows.
Assume that $X$ is built from a discrete time skeleton $\hat X$
such that $X$ is updated according to the successive positions of $\hat X$
after independent exponential times of parameter
\begin{equation} \label{tizio}
	\bar w = \max_{x \in {\cal X}} w(x) = \max_{x \in {\cal X}} \sum_{y \neq x} w(x, y) < + \infty 
	\,,
\end{equation}
so that, with $\hat P$ the transition matrix of $\hat X$,
\begin{equation} \label{caio}
	L = \bar w (\hat P - \mathbbm{1}_{{\cal X}})
	\,.
\end{equation}
The randomized or skeleton return time $T_R^+$ is
\begin{equation} \label{sempronio}
	T_R^+ = \inf\left\{
		t \geq \tau_1 :\:
		X(t) \in R
	\right\}
\end{equation}
with $\tau_1$ the first updating time
(which is an exponential time of parameter $\bar w$).
One always has $T_R \leq T_R^+ \leq T_R^{++}$
and, for $x \in R$, it holds
\begin{equation}
	E_x[T_R^{++}]
	= E_x[T_R^+] + P_x(X(\tau_1) = x) E_x[T_R^{++}]
	\,,
\end{equation}
so that 
\begin{equation} \label{plastica}
	E_x[T_R^{++}]
	= \frac{\bar w}{w(x)} E_x[T_R^+]
	\,,
\end{equation}
with $w(x)$ as in~\eqref{VariableSpeed}.
 
In Appendix~\ref{Wentzell}, Lemma \ref{WentzellThm}, we prove,
with the help of Wilson's algorithm
and elementary Green's function computations,
a formula for $E_x[T_R]$, which is originally due to Freidlin and Wentzell
(Lemma 3.3 in~\cite{FW98}).
We will use this formula to compute the mean value of $E_x[T_R]$
when $R$ is the set of roots sampled from either $\nu_Q$
or ${\mathbb P}\left(\cdot \, \big| |\rho(\Phi_q)| = m\right)$ for any given~$m$.

\begin{theorem}\label{HittingRootsThm}{\bf (Hitting-time formulas)}
For any $x\in \mathcal{X}$
\begin{equation}\label{GeneralHittingRoots}
{\mathbb E}\left[E_x\left[T_{\rho(\Phi_Q)}\right]\right]
=\displaystyle\sum_{y\in\mathcal{X}}\frac{1}{q(y)}\left[{\mathbb P}\left(\rho\left(\tau_x(\Phi_Q)\right)=\{y\}\right)-{\mathbb P}(\rho(\Phi_Q)=\{y\})\right],
\end{equation}
with $\tau_x(\Phi_Q)$ being the unique tree in $\Phi_Q$ containing $x$.
In the standard case, $q(x)\equiv q>0$, equation \eqref{GeneralHittingRoots} reduces to 
\begin{equation}\label{HittingRoots}{\mathbb E}\left[E_x\left[T_{\rho(\Phi_q)}\right]\right]
=\frac{1}{q}\left[1-{\mathbb P}\left(|\rho(\Phi_q)|=1\right)\right]=\frac{1}{q}\left(1-\prod_{i = 1}^{n - 1} \frac{\lambda_i}{q + \lambda_i}\right).
\end{equation}
Moreover, for $m<n$ and any $x\in \mathcal{X}$, 
\begin{equation}\label{ConditionedHittingRoots}
{\mathbb E}\left[E_x\left[T_{\rho(\Phi_q)}\right] \big| |\rho(\Phi_q)|=m\right]
=\frac{{\mathbb P}(|\rho(\Phi_q)|=m+1)}{q{\mathbb P}(|\rho(\Phi_q)|=m)}
= \frac{a_{m + 1}}{a_m}
\,,
\end{equation}
with $a_k$ the coefficient of degree $k$ 
in the characteristic polynomial
\begin{equation}
	\det(q - L) = a_1 q + \cdots + q^n
	\,.
\end{equation}
\end{theorem}

\begin{proof}
Observe first that, for any $R\subset \mathcal{X}$,
\begin{equation}\label{RootSet}
{\mathbb P}(\rho(\Phi_Q)=R)=\frac{\sum_{\phi:\rho(\phi)=R}w_Q(\phi)}{Z_Q}= \frac{Z_R(0)}{Z_Q}{\rm det}_{R\setminus S}(Q),
\end{equation}
with $Z_R(0)$ as in \eqref{Rmeasures}.
By using equation \eqref{RootSet} together with \eqref{WentzellFormula}, we have that
\begin{equation}
\begin{aligned}
{\mathbb E}[E_x[T_{\rho(\Phi_Q)}]]&= \displaystyle\sum_{R \neq \emptyset}{\mathbb P}(\rho(\Phi_Q)=R)E_x[T_R]\\
&=\displaystyle\sum_{R \neq \emptyset :\: R \supset S}\frac{{\rm det}_{R\setminus S}(Q)}{Z_Q}
\displaystyle\sum_{y\notin R}\sum_{
	\stackrel{\scriptstyle\phi:\rho(\tau_x(\phi))=\{y\},}
		{\scriptstyle\rho(\phi)=R\cup\{y\}}
}w(\phi)\\
&=\frac{1}{Z_Q}
\displaystyle\sum_{y\notin S}
\displaystyle\sum_{
	\stackrel{\scriptstyle R \neq \emptyset :\: R \supset S}
		{ R\not\ni y}
}
\displaystyle\sum_{
	\stackrel{\scriptstyle\phi:\rho(\tau_x(\phi))=\{y\},}
		{\scriptstyle\rho(\phi)=R\cup\{y\}}
}\frac{w_Q(\phi)}{q(y)}\\
&=\frac{1}{Z_Q}\displaystyle\sum_{y\notin S}\left[\displaystyle\sum_{\phi:\rho(\tau_x(\phi))=\{y\}}\frac{w_Q(\phi)}{q(y)}
-\sum_{\phi:\rho(\phi)=\{y\}}\frac{w_Q(\phi)}{q(y)}\right]\\
&=\displaystyle\sum_{y\notin S}\frac{1}{q(y)}\left[{\mathbb P}(\rho(\tau_x(\phi_Q))=\{y\})-{\mathbb P}(\rho(\Phi_Q)=\{y\})\right].
\end{aligned}
\end{equation}
The restriction of summing over $y\notin S$ can be removed, since $1/q(y)=0$ for $y\in S$. Hence \eqref{GeneralHittingRoots} holds and 
\eqref{HittingRoots} readily follows when $q(x)\equiv q>0$.
The proof of \eqref{ConditionedHittingRoots} follows by an analogous computation.
\end{proof}

Note that the r.h.s. of \eqref{HittingRoots} and \eqref{ConditionedHittingRoots} is independent of the starting point $x$.
This latter observation allows to answer the questions in the introduction.
In fact, no matter the geometry of the graph and the weights we are considering, 
we can take the random subset $R$ of ${\cal X}$ which is given
by the root set of a random forest with law ${\mathbb P}\left(\cdot \big||\rho(\Phi_q)|=m\right)$,
and the formula in equation \eqref{ConditionedHittingRoots} says that the hitting times do not dependent on the starting point $x$.

To practically obtain a sample from $\nu_q$ with approximately $m \leq n$ roots
when $L$ has only real eigenvalues, one can use Wilson's algorithm and play with the parameter $q$ as follows.
If $q$ is such that
\begin{equation}\label{albicocca}
\sum_{i = 0}^{n - 1} \frac{q}{q + \lambda_i} = m,
\end{equation}
one has an expected number of $m$ roots with fluctuations
of order $\sqrt m$ or smaller (see Corollary~\ref{Bernoullicity}).
In principle one should compute the eigenvalues of $L$, which is in general difficult for large $n$,
and then solve equation (\ref{albicocca}) in $q$.
To overcome this obstacle, a possible alternative procedure is the following:
\begin{enumerate}
\item Start with any positive $q$ and run Wilson's algorithm with this parameter
	to get a sample from $\nu_q$ with a certain number $r$ of roots.
\item Replace $q$ by $q * m / r$ and run again Wilson's algorithm with this
	new parameter to get a new sample with another number of roots,
	say $r$ again.
\item Iterate the previous step until a sample with $r$ roots satisfying
	$m - 2 \sqrt m \leq r \leq m + 2 \sqrt m$ is obtained.
\end{enumerate}
As a matter of fact, $q \mapsto qm / \sum_i q / (q + \lambda_i)$
rapidly converges to the solution of (\ref{albicocca}),
hence the algorithm rapidly reaches an end.
Since we believe this procedure to be quite far from an optimal one,
we are only sketchy on this point.
Jus to give some example, we sampled in this way approximatively 100000, 10000, 1000 and 100 roots
on the $512 \times 512$ grid for the random walk in a Brownian sheet
potential with inverse temperature $\beta = .04$.
We obtained 100443, 10032, 1042 and 111 roots in 8, 6, 6 and 8 iterations, respectively.

In equation~(\ref{HittingRoots}) and~(\ref{ConditionedHittingRoots})
the starting point $x$ is given and does not depend on the root set $\rho(\Phi_q)$.
The next two propositions deal with the mean value of
$E_X\bigl[T^+_{\rho(\Phi_q)}\bigr]$ (recall~(\ref{stella})-(\ref{VariableSpeed}))
when $X$ is uniformly chosen in $\rho(\Phi_q)$
and with the mean value of $\max_{x \in {\cal X}} E_x\left[T_{\rho(\Phi_q)}\right]$.

\begin{proposition}\label{edera}
{\bf{(Return-time in $\rho(\Phi)$)}}
For all $x$ in ${\cal X}$, all $q > 0$ and all positive $m \leq n$, it holds
\begin{equation}\label{return}
	{\mathbb E}\left[
		\mathbbm{1}_{\{x \in \rho(\Phi_q)\}}
		E_x\left[
			T^+_{\rho(\Phi_q)}
		\right]
	\right]
	= {\mathbb E}\left[
		\mathbbm{1}_{\{x \in \rho(\Phi_q)\}}
			E_x\left[
				T^+_{\rho(\Phi_q)}
		\right]
		\Big|\, |\rho(\Phi_q)| = m
	\right]
	= \frac{1}{\bar{w}}
\end{equation}
and
\begin{equation} \label{ciuffo}
	{\mathbb E}\left[
		\mathbbm{1}_{\{x \in \rho(\Phi_q)\}}
		E_x\left[
			T^{++}_{\rho(\Phi_q)}
		\right]
	\right]
	= {\mathbb E}\left[
		\mathbbm{1}_{\{x \in \rho(\Phi_q)\}}
			E_x\left[
				T^{++}_{\rho(\Phi_q)}
		\right]
		\Big|\, |\rho(\Phi_q)| = m
	\right]
	= \frac{1}{w(x)}
	\,.
\end{equation}

In particular,
\begin{equation} \label{returnFromUniform}
	{\mathbb E}\left[
		E_{U(\rho(\Phi_q))}\left[
			T^+_{\rho(\Phi_q)}
		\right] \Big|\,
		|\rho(\Phi_q)| = m
	\right]
	= \frac{1}{\bar{w}} \frac{n}{m}
\end{equation} 
and
\begin{equation} \label{chicca}
	{\mathbb E}\left[
		E_{U(\rho(\Phi_q))}\left[
			T^{++}_{\rho(\Phi_q)}
		\right] \Big|\,
		|\rho(\Phi_q)| = m
	\right]
	= \left(\frac{1}{n} \sum_{x \in {\cal X}} \frac{1}{w(x)}\right) \frac{n}{m}
	\,,
\end{equation} 
where $U(\rho(\Phi_q))$ stands for a random point uniformly distributed in $\rho(\Phi_q)$.
\end{proposition}

\begin{proof}
Along this proof, for an arbitrary subset $R\subset\mathcal{X}$,
we write $\nu(R)$ for either ${\mathbb P}(\rho(\Phi_q)=R)$
or ${\mathbb P}\left(\rho(\Phi_q) = R \,\Big|\, |\rho(\Phi_q)| = m\right)$
and we set $g_{R}(x)=E_x\bigl[\hat{T}_{R}\bigr]= \bar{w}E_x\bigl[T_{R}\bigr]$
for any $x$ in ${\cal X}$.
If $x \in R$, then $g_R(x) = 0$, while, when $x \not \in R$, 
\begin{equation}\label{1}
g_{R}(x)=\sum_{y\in\mathcal{X}}p(x,y)E_x\left[\hat{T}_{R} \big|\, \hat{X}(1) = y\right]=\sum_{y\in\mathcal{X}}p(x,y)(1+g_R(y))=1+(Pg_{R})(x)
\,.
\end{equation} 
Setting $h_{R}(x)={\mathbbm 1}_{\{x \in R\}} E_x\bigl[\hat{T}^+_{R}\bigr]$,
we also have, when $x \in R$,
\begin{equation}\label{2}
h_R(x) = 1+(Pg_{R})(x)
\,.
\end{equation} 
By defining $g(x)=\sum_{R\subseteq\mathcal{X}}\nu(R)g_{R}(x)$ for all $x$ in {\cal X}, we then get 
\begin{equation}\label{3}
\begin{aligned}
(Pg)(x) &=\sum_{R\subseteq\mathcal{X}}\nu(R)(Pg_{R})(x)=
\sum_{R\not\ni x}\nu(R)[g_{R}(x)-1]+\sum_{R\ni x}\nu(R)[h_{R}(x)-1]\\
&= -\sum_{R}\nu(R) + \sum_{R\not\ni x}\nu(R)g_{R}(x)+\sum_{R\ni x}\nu(R)h_{R}(x)\\
&= -1 + g(x) +\sum_{R\ni x}\nu(R)h_{R}(x)
\,,
\end{aligned}
\end{equation} 
where in the last equality we have used that $g_{R}(x)=0$ if $x\in R$, and that $\sum_{R}\nu(R)=1$ since $\nu$ is a probability measure on subsets of $\mathcal{X}$.
Now Theorem \ref{HittingRootsThm} says that
the function $g$ is constant on $\mathcal{X}$, therefore harmonic on $\mathcal{X}$. 
Equivalently, $(Pg)(x)=g(x)$ for any $x\in\mathcal{X}$, which together with 
\eqref{3} implies that 
\begin{equation}\label{4}
\sum_{R\ni x}\nu(R)h_{R}(x)=1. 
\end{equation} 
By recalling the definition of $h_R$ and passing in continuous time, we see that equation \eqref{4} is equivalent to \eqref{return}.  
Equation~\eqref{ciuffo} is then a consequence of~\eqref{plastica}.

Then, the claim in \eqref{returnFromUniform} readily follows by writing 
\begin{equation}
	{\mathbb E}\left[
		E_{U(\rho(\Phi_q))}\left[
			T^+_{\rho(\Phi_q)}
		\right]
		\Big|\, |\rho(\Phi_q)| = m
	\right]
	= {\mathbb E}\left[
		\sum_{x\in\mathcal{X}} \frac{\mathbbm{1}_{\{x\in\rho(\Phi_q)\}}}{m}
		E_{x}\left[
			T^+_{\rho(\Phi_q)}
		\right]
		\bigg|\, |\rho(\Phi_q)| = m
	\right]
	\,, 
\end{equation}
and plugging \eqref{return} in the r.h.s above.
Equation~\eqref{chicca} follows in the same way from~\eqref{ciuffo}.
\end{proof}

\begin{proposition}\label{blu}
{\bf{(Largest expected hitting-time estimates)}}
For all $q > 0$ it holds
\begin{equation}\label{MaxEstimate}
{\mathbb E}\left[\max_{x\in \mathcal{X}}E_x\left[T_{\rho(\Phi_q)}\right]\right]
\leq \frac{1}{q}\left[{\mathbb E}\left[|\rho(\Phi_q)|\right]-{\mathbb P}(|\rho(\Phi_q)|=1)\right]
= \frac{1}{q}\left(\sum_{i = 0}^{n - 1} \frac{q}{q + \lambda_i} - \prod_{i = 1}^{n - 1} \frac{\lambda_i}{q + \lambda_i}\right).
\end{equation}
Furthermore, for any positive $m\leq n$, it holds (recall the notation of Theorem~\ref{HittingRootsThm})
\begin{equation}\label{MaxEstimateFixedRoots}
{\mathbb E}\left[\max_{x\in \mathcal{X}}E_x\left[T_{\rho(\Phi_q)}\right]\Big| \, |\rho(\Phi_q)|=m\right]\leq 
\frac{
	(m+1){\mathbb P} \left(
		|\rho(\Phi_q)| = m + 1
	\right)
}{
	q {\mathbb P} \left(
		|\rho(\Phi_q)| = m 
	\right)
}	
= (m + 1) \frac{a_{m + 1}}{a_m}
\,.
\end{equation}
\end{proposition}
\begin{proof}
As we remarked before, Wilson's algorithm works also when considering more than one absorbing state.
Denote by $T^W_R$ the running time of Wilson's algorithm 
(i.e. the total running time of the loop erased random walks needed to cover the whole graph)
when the absorbing states form a non-empty subset $R$ of $\mathcal{X}$ (this amounts to sample $\nu_{0, R}$).
It can be shown (see e.g.~\cite{M00}, Proposition 1)
that the mean running time\footnote{
	This running time is actually independent of the obtained sample and its law is the same
	as that of a sum of independent exponential variables with parameters $\lambda_{i, R}$
	when these eigenvalues are real.
	The same holds in the case of complex eigenvalues by defining ``the sum of exponential variables''
	through its Laplace transform and the same algebraic formula as in the real case.
} can be expressed in spectral terms as $\sum_{i=0}^{|R|-1}\frac{1}{\lambda_{i,R}}$,
with $\lambda_{i,R}$ being the eigenvalues of the operator $[L]_{R^c}$,
the sub-Markovian generator associated with the process absorbed in $R$.
Note at this point that we can overestimate the l.h.s. of \eqref{MaxEstimate} by the expectation of $T^{W}_{\rho(\Phi)}$.
Hence, using (\ref{lucien}) and looking at the coefficient of degree 1 in $Z_R(q)$,
\begin{equation}
\begin{aligned}
{\mathbb E}\left[\max_{x\in \mathcal{X}}E_x\left[T_{\rho(\Phi_q)}\right]\right]
&\leq {\mathbb E}\left[T^{W}_{\rho(\Phi_q)}\right]= {\mathbb E}\left[\sum_{i=0}^{|\rho(\Phi_q)|-1}\frac{1}{\lambda_{i,\rho(\Phi_q)}}\right]
=\displaystyle\sum_{R \neq \emptyset}\frac{Z_{R}(0)}{Z(q)}q^{|R|}\sum_{i=0}^{|R|-1}\frac{1}{\lambda_{i,R}}
\\&=\frac{1}{Z(q)}\displaystyle\sum_{R\neq \emptyset}q^{|R|}\sum_{\phi:\rho(\phi)\supset R, |\rho(\phi)|=|R|+1}\frac{w_{q,R}(\phi)}{q}
\\&=\frac{1}{qZ(q)}\sum_{k=1}^n\sum_{R : |R|=k}\sum_{\phi:\rho(\phi)\supset R, |\rho(\phi)|=k+1}q^{|R|}w_{q,R}(\phi)
\\&=\frac{1}{qZ(q)}\sum_{k=1}^n\sum_{\phi:|\rho(\phi)|=k + 1}w_q(\phi)(k+1)
\\&=\frac{1}{q}\sum_{\phi:|\rho(\phi)|\geq 2}\frac{w_q(\phi)|\rho(\phi)|}{Z(q)},
\end{aligned}
\end{equation}
and the latter equals the r.h.s. of \eqref{MaxEstimate}.
The bounds in \eqref{MaxEstimateFixedRoots} follows by a similar argument.
\end{proof}

\par\noindent
{\bf Remark:} These estimates can often be a gross overestimation
of the largest expected hitting time as soon as  
$q$ is not very small or very large, or when $m$ is not close to $1$ or $n$. 
It is not difficult however to build examples for which the estimates are tight
for all $q$ and $m$. (One can for example consider a one dimensional random walk with drift.)

%%%%%%%%%%%%%%%%%%%%%%%%%%%%%%%%%%%%%%%%%%%%%%%%%%%%%%%%%%%%%%%%%%%%%%%%%%%%%%%%%%%%%%%%%%%%%%%%%%%%%%%%%%%%%%%%%%%%

\section{Re-reading Micchelli-Willoughby proof}\label{MW}

Throughout this section we work with the Markov process $X$ on $\mathcal{X}$ in \eqref{MarkovGenerator}, 
{\em under the assumption that $X$ is reversible with respect to some probability measure $\mu$ on $\mathcal{X}$},
 i.e., $L$ is a self-adjoint operator in $l^2(\mu)$ endowed with the inner product
\begin{equation}\label{inner product}
\langle f,g\rangle_{\mu}=\sum_{x\in\mathcal{X}}\mu(x)f(x)g(x).
\end{equation}
For $R \subsetneq {\cal X}$, possibly $R = \emptyset$, we turn the points of $R$ into ``absorbing points''
by adding infinite weight edges towards a cemetery $\Delta$.
We denote by $\lambda_{0, R} \leq \lambda_{1, R} \leq \cdots \leq \lambda_{l - 1, R}$,
with $l = n - |R|$,
the eigenvalues of $[-L]_{R^c}$,
and following~\cite{DM09, F09, M10},
we define, for each $x$ in $R^c$,
a sequence of {\em local equilibria}
by setting
\begin{eqnarray}\label{SequenceMeas}
\nu^x_{l - 1} & = & \delta_x,\\
\nu^x_{k - 1} & = & \nu^x_{k} \frac{[L]_{R^c} + \lambda_{k, R}}{\lambda_{k, R}}, \quad 1 \leq k \leq l - 1.  \label{acqua_di_rosa}
\end{eqnarray}
Theorem~3.2 in~\cite{MW79} is a statement on symmetric matrices that in our setting can be described as follows.
\begin{theorem} {\bf (Micchelli and Willoughby~\cite{MW79})}\label{macine}
For all $x$ in ${\cal X}$ and all $k < l$,
$\nu_k^x$ is a non-negative measure.
\end{theorem} 
In this section we give a proof of this theorem following the key steps of Micchelli and Willoughby's algebraic proof, however, unlike the original proof,
we develop probabilistic or combinatorial arguments.

Before starting the proof we note, following~\cite{M10},
that equation (\ref{acqua_di_rosa}) can be rewritten as
\begin{equation}
\nu^x_k [L]_{R^c} = \lambda_{k, R} (\nu_{k - 1}^x - \nu_k^x),
\end{equation}
which gives the following interpretation. 
The process leaves the measure, or ``state'', $\nu^x_k$
at rate $\lambda_{k, R}$ to be absorbed in $R$
or to decay into $\nu^x_{k - 1}$.
This can be turned
into a rigorous mathematical statement~\cite{M00},
provided that $\nu^x_k$ and $\nu^x_{k - 1}$
are indeed non-negative measures,
as claimed in Theorem~\ref{macine}.
Then, by looking at the different decay times up to an exponential
time $T_q$ that is independent from the process, 
and by observing that, by Hamilton-Cayley theorem,
the process leaves the state $\nu^x_0$ at rate
$\lambda_{0, R}$ only to be absorbed in $R$, we get,
for all $x$ and $y$ in $R^c$,
\begin{equation}\label{lampo}
\begin{split}
P_x(X(T_q \wedge T_R) = y) 
= & \frac{q}{q + \lambda_{l - 1, R}} \nu^x_{l - 1}(y)
+ \frac{\lambda_{l - 1, R}}{q + \lambda_{l - 1, R}} \frac{q}{q + \lambda_{l -2, R}} \nu^x_{l - 2}(y)\\
& + \cdots 
+ \frac{\lambda_{l - 1, R}}{q + \lambda_{l - 1, R}} \cdots \frac{\lambda_{2, R}}{q + \lambda_{2, R}} \frac{q}{q + \lambda_{1, R}} \nu^x_{1}(y)\\
& + \frac{\lambda_{l - 1, R}}{q + \lambda_{l - 1, R}} \cdots \frac{\lambda_{1, R}}{q + \lambda_{1, R}} \frac{q}{q + \lambda_{0, R}} \nu^x_{0}(y).
\end{split}
\end{equation}
The left hand side in (\ref{lampo}) is the probability to have $\rho(\tau_x(\Phi_{q, R})) = \{y\}$.
Then, multiplying by $Z_R(q) = \prod_i (q + \lambda_{i, R})$ (recall (\ref{lucien})), dividing by $q$, 
and denoting the result by $W_R(q)(x, y)$, we have on the right hand side
\begin{equation} \label{nodo}
	W_R(q)(x, y)
	=  Z_R(q) [q\mathbbm{1}-L]_{R^c}^{-1}(x, y) 
\end{equation}
or, equivalently, 
\begin{equation} \label{capello} 
	W_R(q)(x, y)
	= \frac{1}{q} \sum_{
		\stackrel{\scriptstyle \phi :\: \rho(\tau_x(\phi)) = \{y\}} 
		{\scriptstyle \rho(\phi) \supseteq R}
	}
	q^{|\rho(\phi)| - |R|} w(\phi) 
	\,,
\end{equation}
while equation~\eqref{lampo} now reads
\begin{equation} \label{capra}
	\begin{split}
		W_R(q)(x, y) = 
		& (q + \lambda_{0, R}) \cdots (q + \lambda_{l - 2, R}) \nu^x_{l - 1}(y) \\ 
		& + (q + \lambda_{0, R}) \cdots (q + \lambda_{l - 3, R}) \lambda_{l - 1, R}\nu^x_{l - 2}(y) \\
		& + \cdots + (q + \lambda_{0, R}) \lambda_{l - 1, R} \lambda_{l - 2, R}\cdots \lambda_{2, R} \nu^x_{1}(y) \\
		& + \lambda_{l - 1, R} \cdots \lambda_{1, R} \nu^x_0(y)
		\,. 
	\end{split}
\end{equation}
Next, in order to prove Theorem~\ref{macine}, we can restrict ourselves, by density and continuity, to the case of distinct eigenvalues.
Then equation \eqref{capra} suggests, for $k < l$,
the following relation for the divided differences (see Definition \ref{papelito} in Appendix~\ref{DividedDiff}) of $W_R$:
\begin{equation}\label{bruxelles}
W_R[-\lambda_{0, R}, \dots, - \lambda_{k, R}](x, \cdot) = \lambda_{l-1, R} \cdots \lambda_{k + 1, R} \nu^x_k
= \delta_x \prod_{i = k + 1}^{l - 1} ([L]_{R^c} + \lambda_{i, R}),
\end{equation}
that is
\begin{equation}\label{cavolo}
W_R[-\lambda_{0, R}, \dots, - \lambda_{k, R}] = \prod_{i = k + 1}^{l - 1} ([L]_{R^c} + \lambda_{i, R}).
\end{equation}
It is worth to stress at this point we have just seen 
that equation \eqref{cavolo} would be {\em a consequence}
of Theorem~\ref{macine}, but our goal is {\em to prove} Theorem~\ref{macine}.
This is what we are ready to do now
by following the main steps of Micchelli and Willoughby's proof.

\smallskip\par\noindent
{\bf Step 1: Checking equation (\ref{cavolo})} (without assuming Theorem~\ref{macine}\dots).
\newline
We simply use Definition \ref{DivDiff2} and spectral decomposition.
With $\mu_i$ being the right eigenvector associated with $\lambda_{i, R}$,
and, for any measure $\nu$, $\langle \mu_i, \nu \rangle = \sum_{x \not \in R} \mu_i(x) \nu(x) / \mu(x)$,
we have for any~$q$ and with $W_R$ defined by \eqref{nodo},
\begin{equation}
\nu W_R(q)=\sum_{i=1}^{n-1} \langle \mu_i,\nu\rangle \prod_{i\neq j}(q +\lambda_{i, R})\mu_j.
\end{equation}This gives
\begin{equation}
\begin{aligned}
\nu W_R[-\lambda_{0, R}\ldots,-\lambda_{k, R}]&
=  \sum_{r=0}^k\frac{\nu W_R(-\lambda_{r, R})}{\prod_{m\neq r}(\lambda_{m, R}-\lambda_{r, R})}=
 \sum_{r=0}^k \sum_{j=1}^{l-1} \langle \mu_j,\nu\rangle \frac{\prod_{i\neq j}(\lambda_{i, R}-\lambda_{r, R})\mu_j}{\prod_{m\neq r}(\lambda_{m, R}-\lambda_{r, R})}\\
&=\sum_{r=0}^k \langle \mu_r,\nu\rangle \frac{\prod_{i\neq j}(\lambda_{i, R}-\lambda_{r, R})\mu_r}{\prod_{m\neq r}(\lambda_{m, R}-\lambda_{r, R})}=
\sum_{r=0}^k \langle \mu_r,\nu\rangle \prod_{i=k+1}^{l-1}(\lambda_{i, R}-\lambda_{r, R})\mu_r\\
&=\sum_{r=0}^{l-1} \langle \mu_r,\nu\rangle \prod_{i=k+1}^{l-1}(\lambda_{i, R}-\lambda_{r, R})\mu_r=\nu \prod_{i=k+1}^{l-1}([L]_{R^c}+\lambda_{i, R})
\end{aligned}
\end{equation}
and, by arbitrariness of $\nu$, equation \eqref{cavolo} readily follows.

\smallskip\par\noindent
{\bf Step 2: A combinatorial identity.}
\newline
The key point of the proof lies in the following lemma.

\begin{lemma}\label{MW2}
For any $x\neq y$ in $\mathcal{X}\setminus R$,
\begin{equation}\label{Wxy}
W_R(q)(x,y)=
w(x,y)Z_{R \cup \{x,y\}}(q)+\sum_{z,z'\in {\cal X} \setminus (R \cup \{x,y\})} w(x,z)W_{R \cup \{x,y\}}(q)(z,z')w(z',y)
\end{equation}
In addition one has
\begin{equation}\label{mirtillo}
W_R(q)(x,x)=Z_{R \cup \{x\}}(q).
\end{equation}
\end{lemma}
\begin{proof}

Let us first consider the case $x\neq y$. Due to \eqref{capello}, we have that
\begin{equation}\label{W1}
 W_R(q)(x,y)
 = \sum_{
	\stackrel{\scriptstyle \phi :\: \rho(\tau_x(\phi)) = \{y\}} 
	{\scriptstyle \rho(\phi) \supseteq R}
}
q^{|\rho(\phi)| - 1 - |R|} w(\phi) 
\end{equation}
We also have
\begin{equation}
\label{W2}
Z_{R \cup \{x,y\}}(q)=\sum_{\phi':\rho(\phi')\supseteq R \cup \{x,y\}}q^{|\rho(\phi')|-2- |R|}w(\phi')
\end{equation}
and 
\begin{equation}
\label{W3}
W_{R \cup \{x,y\}}(q)(z,z')
 = \frac{1}{q} \sum_{
	\stackrel{\scriptstyle \phi'' :\: \rho(\tau_z({\phi''})) = \{z'\}} 
	{\scriptstyle \rho(\phi) \supseteq R \cup \{x, y\}}
}
q^{|\rho(\phi)| - 2 - |R|} w(\phi) 
 = \sum_{
	\stackrel{\scriptstyle \phi'' :\: \rho(\tau_z({\phi''})) = \{z'\}} 
	{\scriptstyle \rho(\phi) \supseteq R \cup \{x, y\}}
}
q^{|\rho(\phi)| - 3 - |R|} w(\phi) 
.
\end{equation}
Next, define for each $\phi$ in (\ref{W1}),
$\phi' = \phi \setminus \{(x, y)\}$ if $(x,y)$ belongs to $\phi$,
and $\phi'' = \phi \setminus \{(x, z);(z', y)\}$ if
$x$ is connected in $\phi$ to $y$ through $z$ and $z'$
(possibly with $z = z'$)
in such a way that $(x,z)$ and $(z', y)$ belong to $\phi$.
Finally, by observing that $|\rho(\phi')|=|\rho(\phi)|+1$ and $|\rho(\phi'')|=|\rho(\phi)|+2$, 
\eqref{Wxy} is obtained from \eqref{W1}, \eqref{W2} and \eqref{W3}.
Then, equation (\ref{mirtillo}) follows from (\ref{W1}) for $y = x \not \in R$.
\end{proof}

\smallskip\par\noindent
{\bf Step 3: Conclusion by induction with Cauchy interlacement theorem.}
\newline For $l \geq 0$, let $\mathcal{P}[l]$ be the following statement:
\begin{quote}
	For all $R \subset {\cal X}$ such that $|R^c| = l$,
	for all $L \geq l$,
	for all $\xi_0 > \xi_1 > \cdots > \xi_L$ such that
	$\xi_i \geq -\lambda_{i, R}$ for all $i < l$,
	for all $k \leq L$, and for all $x, y \not \in R$:
	\begin{equation*}
		W_R[\xi_0, \dots, \xi_k](x, y) \geq 0.
	\end{equation*}
\end{quote}

We can proceed  inductively to show that $\mathcal{P}[l]$ holds.

For $l = 0, 1$, the claim is obvious.
Fix $l \geq 2$.
In the case $x = y$, the inductive hypothesis is unnecessary. 
Indeed, from (\ref{mirtillo}), one has 
\begin{equation}
	W_R[\xi_0, \dots, \xi_k](x, x)  = Z_{R \cup \{x\}}[\xi_0, \dots, \xi_k].
\end{equation}
Then note that, by Cauchy interlacement theorem, $\xi_i \geq -\lambda_{i, R}$ implies that
$\xi_i \geq -\lambda_{i, R \cup \{x\}}$ for $i < l - 1$,
and hence, by Lemma \ref{DivDif}, we get
\begin{equation}
	W_R[\xi_0, \dots, \xi_k](x, x) \geq 0. 
\end{equation}
When $x \neq y$, $\mathcal{P}[l]$ follows in the same way by using (\ref{Wxy})
and the inductive hypothesis.

We can finally conclude the proof of Theorem~\ref{macine}.
It suffices to apply the claim with $\xi_i = -\lambda_{i, R}$ for all $i < l$
and recall (\ref{cavolo}) or (\ref{bruxelles}).

%%%%%%%%%%%%%%%%%%%%%%%%%%%%%%%%%%%%%%%%%%%%%%%%%%%%%%%%%%%%%%%%%%%%%%%%%%%%%%%%%%%%%%%%%%%%%%%%%%%%%%%%%%%%%%%%%%%%

\section{Rooted partitions, coalescence and fragmentation}\label{marmelata}

In this section we present two coalescence and fragmentation
processes closely related with our forest measures.
The first one is obtain by coupling together all our forest measures $\nu_q$.
The second one admits $\nu_q$ as invariant measure
and gives some information on the ``rooted partition''
which is naturally associated with each spanning rooted forest.

\subsection{Coupling the forest measures for different values of $q$.}

The coupling we present can be seen as a coalescence and fragmentation process when 
$q$ decreases to $0$ and $t = 1/q$ is thought as time.
The main idea is to make use of Wilson's original representation of his algorithm
with ``site-indexed random paths'' which we present in the next two paragraphs.

\subsubsection{Random walk: stack representation}
Assume that, to each site of the graph, is attached
an infinite list or collection of arrows pointing towards one neighbour, and that these arrows are independently
distributed according to the discrete skeleton 
transition probabilities $\hat P$ as defined by equations~\eqref{tizio}-\eqref{caio}.
In other words, an arrow, pointing towards the neighbour $y$ 
of a site $x$, appears at each level in the list associated with $x$
with probability $\hat P(x, y) = w(x, y) / \bar w$ (in this context, we set $w(x, x) = \bar w - \sum_{y \neq x} w(x, y)$,
and consider $x$ itself as one of its possible neighbours).
Imagine that each list of arrows attached to a site is piled down
in such a way that it make sense to talk of
an infinite stack with an arrow on the top of this stack.
By using this representation, one can generate the random walk on the graph as follows.
At each jump time, the random walk steps to the neighbour pointed
by the arrow on the top of the stack where the walker is sitting,
and the top arrow is erased from the stack.
This representation is often referred to
as the Diaconis-Fulton representation (see~\cite{DF}).

\subsubsection{Wilson's algorithm: stack representation}
To describe Wilson's algorithm one has to introduce a further ingredient: pointers
to the absorbing state $\Delta$ in each stack.
Such a pointer should appear with probability $q / (\bar w + q)$
at each level in the stack. One way to introduce it
is by generating independent uniform random variables $U$
together with each original arrow in the stack, and by replacing
the latter by a pointer to the absorbing state
whenever $U < q / (\bar w + q)$.

A possible description of Wilson's algorithm is then the following.
\begin{itemize}
\item[i.]
Start with a particle on each site. 
Both particles and sites will be declared either {\em active} or {\em frozen}. 
At the beginning all sites and particles are declared to be active.
\item[ii.]
Choose an arbitrary particle among all the active ones
and look at the arrow at the top of the stack it is seated on. Call $x$ the site where the particle is seated.
\begin{itemize}
\item If the arrow is the pointer towards $\Delta$, declare the particle to be frozen and site $x$ as well.
\item If the arrow points towards another site $y \neq x$, remove the particle and keep the arrow. 
	We say that this arrow is now {\em uncovered}. 
\item If the arrow points to $x$ itself, remove the arrow.
\end{itemize}
\item[iii.]
Once again, choose an arbitrary particle among all the active ones, look at the arrow on the top of the stack it is seated on, and call
$x$ the site where the particle is seated.
\begin{itemize}
\item
If the arrow points to $\Delta$, the particle is declared to be frozen,
and so are declared $x$ and all the sites eventually leading to $x$ by following uncovered top pile arrow paths.
\item
If the arrow points to a frozen site, remove the chosen particle at $x$,
keep the (now uncovered) arrow, and 
freeze site $x$ as well as any site eventually leading to $x$ by following uncovered top pile arrow paths.
\item 
If the arrow points to an active site, then there are two possibilities.
By following the uncovered arrows at the top of the stacks, we either reach 
a different active particle, or run in a loop back to $x$.
In the former case, remove the chosen particle from site $x$ and keep the discovered arrow.
In the latter, erase all the arrows along the loop and put an active particle
on each site of the loop.
Note that this last case includes the possibility for the discovered arrow of pointing to $x$ itself,
in which case, we just have to remove the discovered arrow.
\end{itemize}
\item[iv.]
Iterate the previous step up to exhaustion of the active particles.
\end{itemize}
The crucial observation is that, no matter of the choice of  the particles at the beginning of the steps, when this algorithm stops,
the same arrows are erased and {\em the same spanning forest of uncovered arrows
and with a frozen particle at each root is obtained.}
In particular, by choosing at each step the last encountered active particle, or the same as in the previous step when we just erased a loop,
we perform a simple loop-erased random walk up to freezing.

\subsubsection{Coupling and sampling} \label{giallo}
Since $\nu_q$ can be sampled using Wilson's algorithm as described above, and the same uniform variables $U$ can be used for each~$q$,
this provides a coupling among all the $\nu_q$'s.
By means of this coupling, one can actually sample $\nu_{q_2}$ starting from a sample
of $\nu_{q_1}$ for $q_2 < q_1$.
Let us now explain this fact.
Note first that, running this algorithm for sampling $\nu_{q_2}$,
one can reach at some point the final configuration obtained
for $\nu_{q_1}$ with the only difference that some frozen particles
of the final configuration obtained with parameter $q_1$ can still be
active at this intermediate step of the algorithm run with $q_2$.
It suffices, indeed, to choose the sequence of active particles
in the same way with both parameters. This is possible since each pointer to $\Delta$
in the stacks with parameter $q_2$ is associated with a pointer to $\Delta$
at the same level in the stacks with parameter~$q_1$.
Thus, to obtain a sample of $\nu_{q_2}$ from a sample of $\nu_{q_1}$,
we just have to replace some frozen particle of the configuration 
sampled with $\nu_q$ and continue the algorithm with parameter $q_2$.
To decide which particle in $\rho(\Phi_{q_1})$ has to be unfrozen or not 
we can proceed as follows.
With probability 
\begin{equation}
\label{dormi}
	p = P\left(
		U < \frac{q_2}{\bar w + q_2}
		\:\bigg|\: U < \frac{q_1}{\bar w + q_1}
	\right)
	= \frac{q_2 (q_1 + \bar w)}{q_1 (q_2 + \bar w)}
	\,,
\end{equation}
each particle in $\rho(\Phi_{q_1})$, independently from each other, is kept frozen.
With probability $1 - p$ a particle in a site $x$ in $\rho(\Phi_{q_1})$ is declared active,
$x$ is also declared active and we set at the top of the pile in $x$ an arrow
that points toward $y$ with probability $w(x, y) / \bar w$.

\subsubsection{Coalescence and fragmentation
process: trajectories}
When $q = 1 / t$ continuously decreases, we obtain a coalescence-fragmentation
process $t \mapsto \Phi_{1 / t}$ in which each tree can fragment and partially coalesce with the other
trees of the forest. When a root of a tree turns active, the tree is eventually fragmented
into a forest, some trees of which being
possibly ``grafted" on the previous frozen trees.
It is worth noting that, by \eqref{swizzera}, the mean number of trees is decreasing along
this coalescence-fragmentation process.

The previous observations show that we can sample the ``finite dimensional distributions''
of the process, i.e. the law of $(\Phi_{1 / t_1}, \Phi_{1 / t_2}, \dots, \Phi_{1 / t_k})$ for any choice of
$0< t_1 < t_2 < \dots < t_k$.
We can actually sample whole trajectories $(\Phi_{1 / t})_{0 \leq t \leq T}$
for any finite $T$. In fact, note first that
at each time $t = 1 / q$, the next frozen particle (or root) becoming active
is uniformly distributed among all the roots, and  
the time $\sigma$ when it ``wakes up" is 
such that the variable
\begin{equation}
	V 
	= \frac{1 / \sigma}{\bar w + 1 / \sigma}
	\label{savona}
\end{equation}
has the law of the maximum of $m$ independent
uniform variables on $[0, q/(\bar w + q))$,
with $m$ being the number of roots at time $t$.
Since, for all $v < q / (\bar w + q)$,
\begin{equation}
	{\mathbb P}(V < v)
	= \left(
		\frac{v}{q / (\bar w + q)}
	\right)^m
	= \left(
		\frac{v(\bar w + q)}{q}
	\right)^m
	,
\end{equation}
$V$ has the same law as $q U^{1/m} / (\bar w + q)$, with $U$ uniform on $[0, 1)$.
Using (\ref{savona}) we can then sample $\sigma$ by setting
\begin{equation}
	\sigma 
	= \frac{\bar w + q - q U^{1 / m}}{q \bar w U^{1 / m}}
	= t \; \frac{\bar w + (1 - U^{1 / m}) / t}{\bar w U^{1 / m}}.
	\label{cornice}
\end{equation}

Summing up, in order to sample the whole trajectory it suffices to proceed as follows
once~$\Phi_{1 / t}$ is sampled at a given jump time $t$: 
\begin{itemize}
\item Choose uniformly a root $x$ in $\rho(\Phi_{1 / t})$.
\item Sample the next jump time $\sigma$
	from a uniform random variable $U$ on $[0, 1)$,
	by using (\ref{cornice}).
\item\label{sbilancia} Restart the algorithm with parameter $1 / \sigma$
	by declaring active the particle in $x$ and putting an arrow to $y$ with 
	probability $w(x, y)/ \bar w$.
\end{itemize}

The next proposition
characterizes the law
of the process $t \mapsto \rho\left(\Phi_{1 / t}\right)$.
\begin{proposition}
	If $\Xi_t = \rho(\Phi_{1 / t})$ for all $t \geq 0$,
	then, for any $A_1, \dots, A_k, A_{k + 1} \subset {\cal X}$
	and any $t_1 < \cdots < t_k < t_{k + 1}$,
	\begin{equation}\label{ciuccio}
		\begin{split}
			{\mathbb P} & \left(
				 A_{k + 1} \subset \Xi_{t_{k + 1}}
				\Big |
				A_k \subset \Xi_{t_k},
				\dots,
				A_1 \subset \Xi_{t_1}
			\right) \\
			& = \sum_{R_1 \subset B_1} \dots \sum_{R_k \subset B_k}
			\left(
				\prod_{i = 1}^k
				\left(
					\frac{t_i}{t_{k + 1}}
				\right)^{|R_i|}
				\left(
					1 - \frac{t_i}{t_{k + 1}}
				\right)^{|B_i \setminus R_i|}
			\right)
			{\rm det}_{A_{k + 1}} K_{q_{k + 1}, R}
		\end{split}
	\end{equation}
with $B_i = A_i \setminus (A_{i + 1} \cup \dots \cup A_k)$ for all $i \leq k$,
$R  = R_1 \cup \dots \cup R_k$
and $q_{k + 1} = 1 / t_{k + 1}$.
\end{proposition}
\begin{proof}
We first prove~\eqref{ciuccio} in the case $k = 1$,
when $B_1 = A_1$.
The observations
made in Section~\ref{giallo} imply that
as far as the event $\left\{A_2 \subset \Xi_{t_2}\right\}$
is concerned,
conditioning on
$\left\{
	A_1 \subset \Xi_{t_1}
\right\}$
is nothing but conditioning 
on the value of the uniform
random variables at the top
of the stacks in $A_1$.
Using the previous terminology,
if we keep `frozen' each site in $A_1$
with probability $p$ defined in~\eqref{dormi}
and we call $B$ the set of the remaining
frozen sites, $\Xi_{t_2} \setminus B$ will
{\em not} be a determinantal process with kernel
$K_{q_2, B}$.
Indeed, the waking up procedure we described
after~\eqref{dormi} implies a bias on the distribution
of the top pile arrow at the unfrozen sites: 
such an arrow cannot be replaced by a pointer to $\Delta$.
To recover a determinantal process with kernel $K_{q_2, B}$,
the set $B$ has to be built by keeping frozen each site in $A_1$
with a smaller probability $p'$ solving
\begin{equation} \label{cane}
	p = p' + (1 - p') \frac{q_2}{q_2 + \bar w}
	\,.
\end{equation}
In such a way the top pile arrow at each unfrozen site
can still be replaced by a pointer to $\Delta$
with probability $q_2 / (q_2 + \bar w)$,
and~\eqref{cane} implies that we 
recover the correct biased probability.
Solving~\eqref{cane} gives $p' = q_2 / q_1 = t_1 / t_2$
and~\eqref{ciuccio} follows.

When $k$ is larger than 1,
the formula is simply obtained
by keeping frozen each site $x$ in $\cup_{i \leq k} A_i$
with a probability that depends on the largest $i$
such that $x \in A_i$. 
This is the reason why the sets $B_i$ are introduced:
$i^*$ is the largest $i$ such that $x \in A_i$,
if and only if $x \in B_{i^*}$.
\end{proof}

\begin{figure}[hbtp]
	\caption{
		\label{farfalla}
		Snapshots of the coalescence and fragmentation process on the torus
		at times $t = 1/q$ equal to 0, .5, 2, 8, 32, 128, 512, \dots, 524288.
		Roots are red, non-root vertices at the border of trees are cyan, 
		other vertices are blue.
	}
	\begin{center}
		\includegraphics[width=120pt]{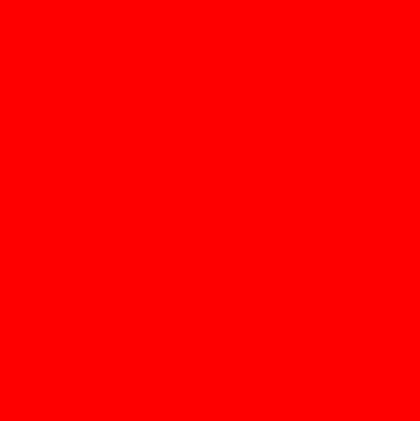}\hspace{20pt}%
		\includegraphics[width=120pt]{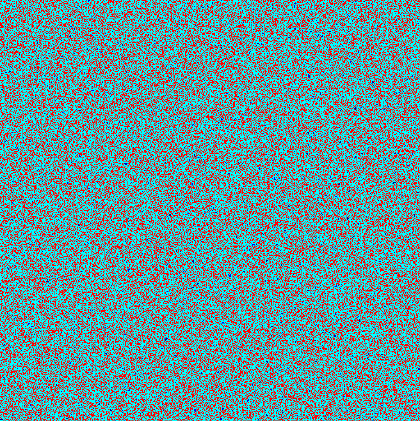}\hspace{20pt}%
		\includegraphics[width=120pt]{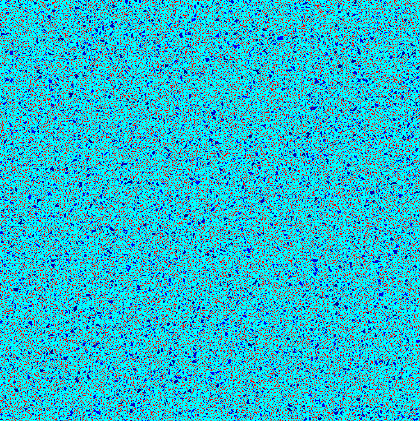}

		\vspace{20pt}
		\includegraphics[width=120pt]{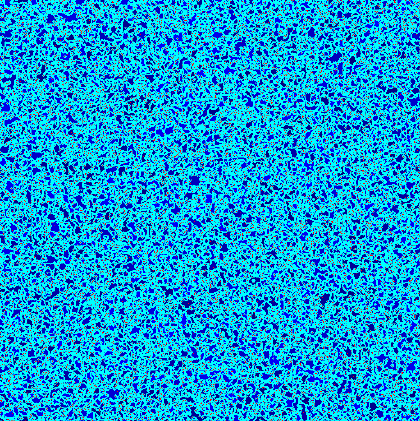}\hspace{20pt}%
		\includegraphics[width=120pt]{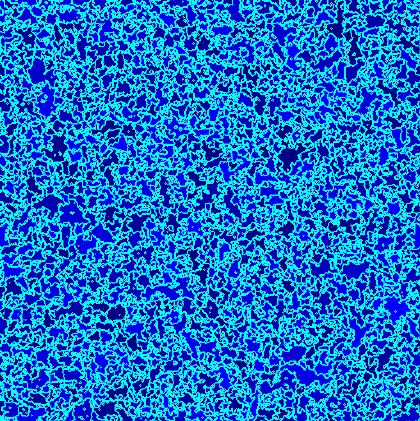}\hspace{20pt}%
		\includegraphics[width=120pt]{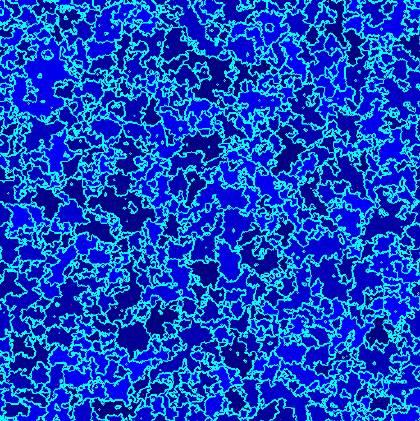}

		\vspace{20pt}
		\includegraphics[width=120pt]{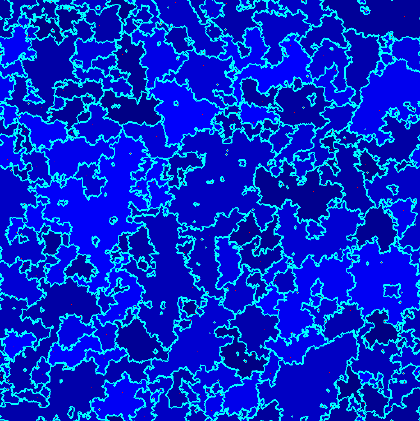}\hspace{20pt}%
		\includegraphics[width=120pt]{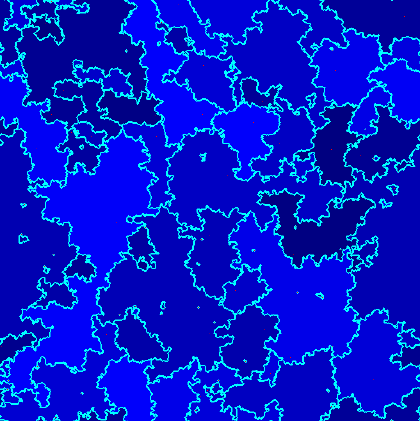}\hspace{20pt}%
		\includegraphics[width=120pt]{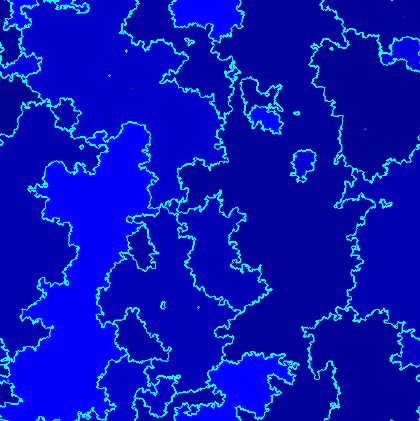}

		\vspace{20pt}
		\includegraphics[width=120pt]{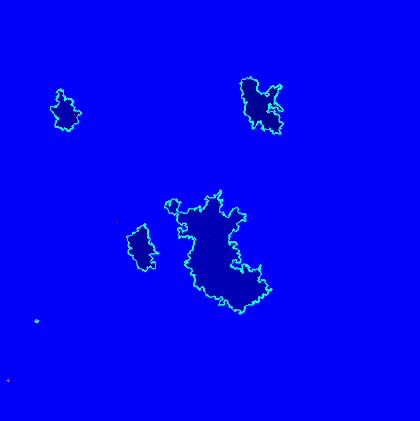}\hspace{20pt}%
		\includegraphics[width=120pt]{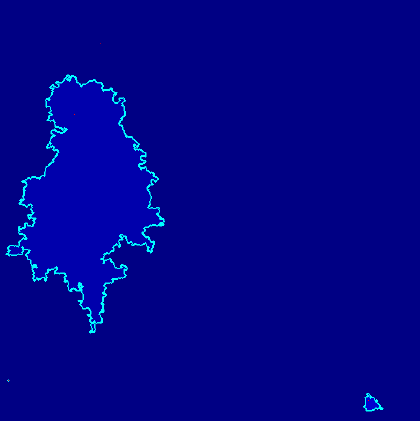}\hspace{20pt}%
		\includegraphics[width=120pt]{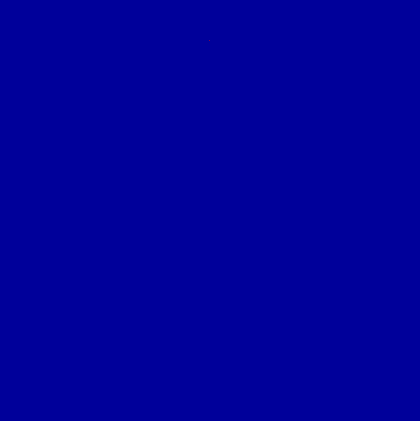}
	\end{center}
\end{figure}

\begin{figure}[hbtp]
	\caption{
		\label{lampada}
		Snapshots 
		at times $t = 1/q$ equal to 0, .5, 2, 8, 32, 128, 512, \dots, 524288
		of the coalescence and fragmentation process on the square grid
		for the random walk in a Brownian sheet potential with inverse temperature
		$\beta = .16$. 
	}
	\begin{center}
		\includegraphics[width=120pt]{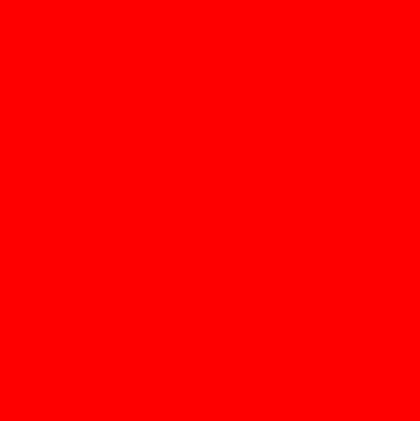}\hspace{20pt}%
		\includegraphics[width=120pt]{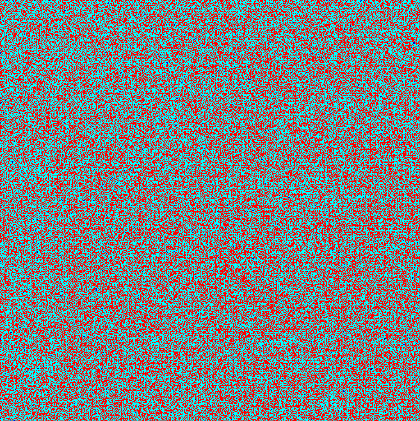}\hspace{20pt}%
		\includegraphics[width=120pt]{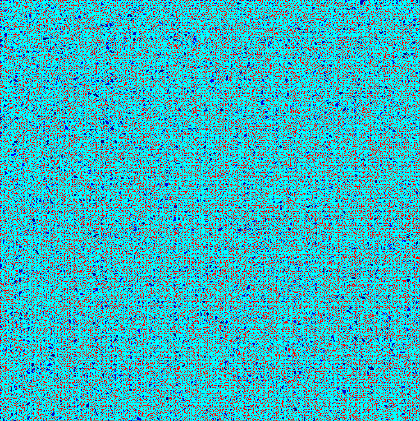}

		\vspace{20pt}
		\includegraphics[width=120pt]{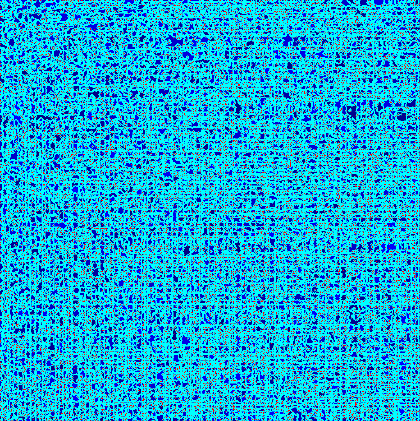}\hspace{20pt}%
		\includegraphics[width=120pt]{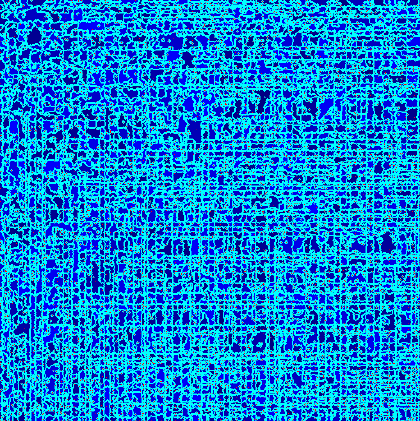}\hspace{20pt}%
		\includegraphics[width=120pt]{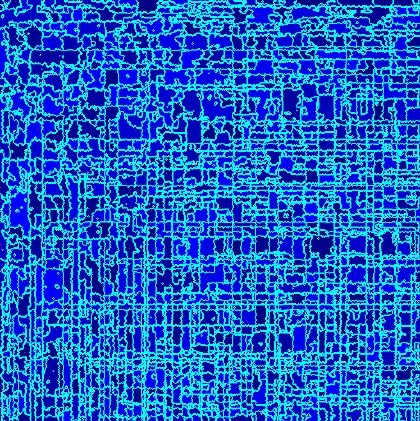}

		\vspace{20pt}
		\includegraphics[width=120pt]{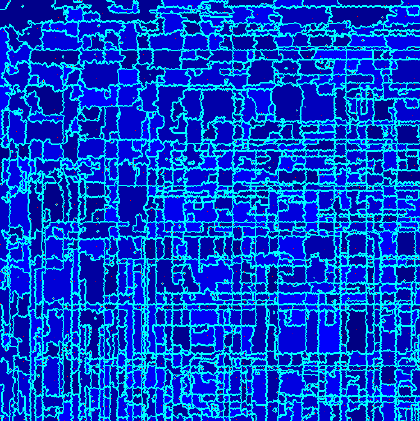}\hspace{20pt}%
		\includegraphics[width=120pt]{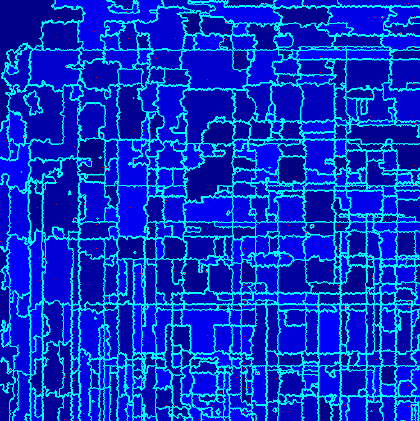}\hspace{20pt}%
		\includegraphics[width=120pt]{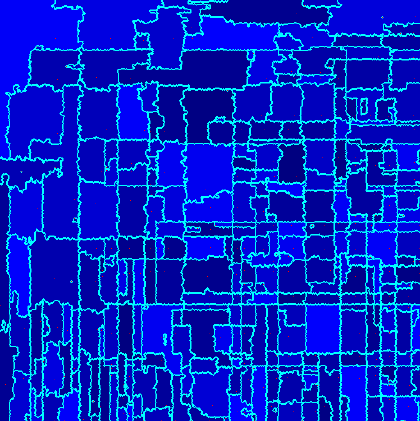}

		\vspace{20pt}
		\includegraphics[width=120pt]{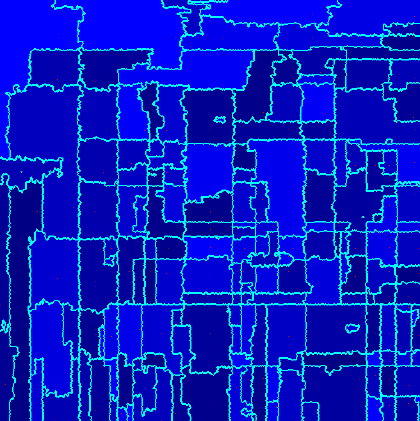}\hspace{20pt}%
		\includegraphics[width=120pt]{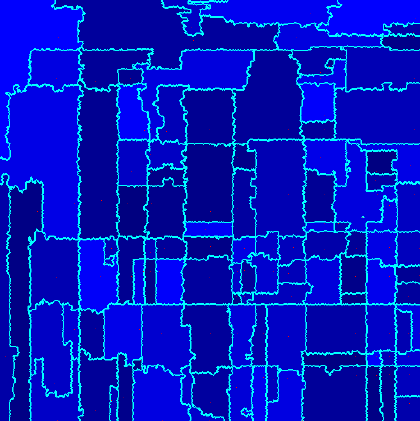}\hspace{20pt}%
		\includegraphics[width=120pt]{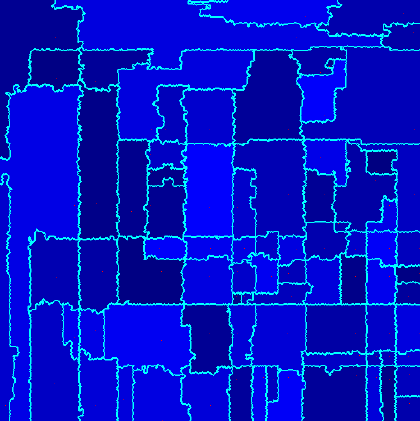}
	\end{center}
\end{figure}

\begin{figure}[htbp]
	\caption{
		\label{sole}
		Tree number in function of time for coalescence and fragmentation
		processes of Figures~\ref{farfalla} and~\ref{lampada} in double logarithmic scale.
	}
	\begin{center}
		\includegraphics{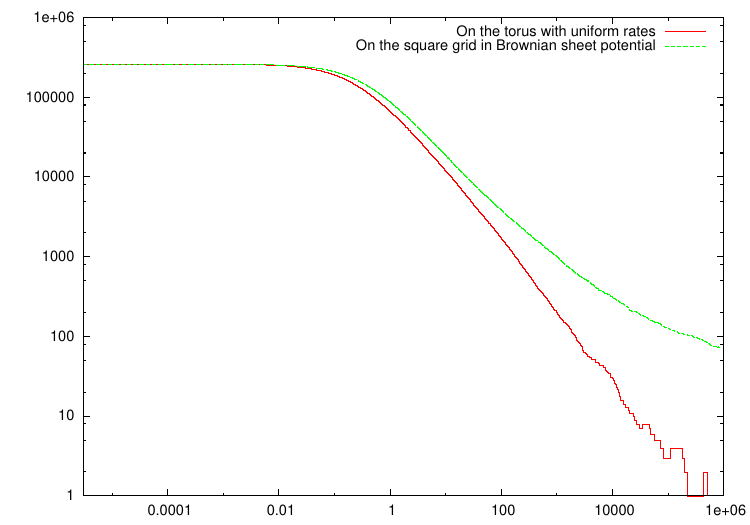}
	\end{center}
\end{figure}

We conclude this section with some open questions.
\begin{enumerate}
\item
	It is a fact that $t \mapsto \Phi_{1 / t}$ ``crosses" almost surely
	all the manifolds
	\begin{equation}
		{\cal F}_m = \{ \phi \in {\cal F} :\: |\rho(\phi)| = m \}
		\,.
	\end{equation}
	Indeed, by (\ref{swizzera}), it starts from ${\cal F}_n$
	and reaches ${\cal F}_1$ almost surely, and, each
	time the number of roots decreases, it does so by only 1
	unit: when the ``activated" tree fragments
	into trees that coalesce only with the previously frozen ones.
	With 
	\begin{equation}
		T_m = \min \{t \geq 0 :\: \Phi_{1 / t} \in {\cal F}_m\}
		\,,
	\end{equation}
	it is also simple to sample $\Phi_{1 / T_m}$.
	But the measure ${\mathbb P}\left( \Phi_q \in \cdot \,\big| |\rho(\Phi_q)| = m\right)$
	is not the law of $\Phi_{1 / T_m}$.
	Is there however a way to use that process to sample this measure?
\item
	One can use that process to estimate $\sum_i \frac{1}{1 + t \lambda_i}$ for all 
	large enough $t$, since this is the expected root number of $|\rho(\Phi_{1 / t})|$.
	Is it then possible to use it to estimate the spectrum of $-L$, 
	or at least the higher part of it in an efficient way ?
\item
	Is it possible to characterize the law of the rooted partition process associated
	with $t \mapsto \Phi_{1 / t}$ like we described the law of $t \mapsto \rho(\Phi_{1 / t})$ ? 
	This partition is the one for which two sites $x$ and $y$ lie in the same component at time $t$
	when they are covered by the same tree in $\Phi_q$ with $q = 1 / t$.
	It is `rooted' since a special point, the root of the tree, is associated
	with each component of the partition.
	We know very little on the law of this partition for each given $q$
	(see next section).
	An easier question would be that of describing the forest process 
	$t \mapsto \Phi_{1 / t}$ itself, since we know that 
	$\Phi_q$ can be described as a determinantal process for each $q$
	(see Section~\ref{archi}).
\end{enumerate}

\subsection{Rooted partition and the forest measure as invariant measure of another coalescence and fragmentation process}
\label{Dynamics}

The other dynamics we want to mention
is a simple variant of the tree random
walk introduced in~\cite{AT89} to prove
the so called Markov chain tree theorem. 
For fixed $q$, the dynamics we now present is another
coalescence and fragmentation process
for which the standard measure $\nu_q$
is the stationary probability measure.

Remind that for a given forest $\phi\in \mathcal{F}$  and $x\in \mathcal{X}$, 
we denote by $\tau_x(\phi)$ the unique tree in $\phi$ containing $x$, and
note that if $e_-\in\rho(\phi)$ then $e\notin\phi$. 
Our dynamics can then be defined as follows.
\begin{definition}{\bf{(Forest Dynamics)}}
\label{RootDynamics}
Fix $q\in[0,\infty)$. Let $\psi$ be the Markov process with state space $\mathcal{F}$ 
characterized by the following formula for the generator ${\cal G}$
acting on functions $f:\mathcal{F}\rightarrow \mathbb{R}$:
\begin{equation}\label{ForestGen}
(\mathcal{G}f)(\phi)=\sum_{e\in\mathcal{E}}\gamma(\phi,e)[f(\phi^e)-f(\phi)], 
\end{equation}
where the transition rate $\gamma(\phi,e)$ and the new state $\phi^e$ are defined as follows:
\begin{enumerate}
\item \label{add}
	If $e_-\in\rho(\phi)$ and $e^+\notin \tau_{e_-}(\phi)$, then 
	$\gamma(\phi,e)=w(e_-,e^+)$ and $\phi^e=\phi\cup\{e\}$.
\item \label{swap}
	If $e_-\in\rho(\phi)$ and $e^+\in \tau_{e_-}(\phi)$, then
	$\gamma(\phi,e)=w(e_-,e^+)$ and $\phi^e=\phi\cup\{e\}\setminus\{e'\}$, 
	with $e'$ being the unique edge in $\phi$ such that $e_-'=e^+$.
\item \label{remove}
	If $e\in\phi$, then
	$\gamma(\phi,e)=q$ and $\phi^e=\phi\setminus\{e\}$.
\item $\gamma(\phi,b)=0$ else.
\end{enumerate}
\end{definition}
The rules corresponding to 1, 2 and 3 can be rephrased by saying 
that we \emph{add, swap} and \emph{remove} a bond from the forest $\phi$, respectively.  
Notice that such a dynamics induces a non-conservative dynamics on the set of roots. 
In particular, when transition 1 occurs, $|\rho(\phi^e)|=|\rho(\phi)|-1$ and 
two trees merge into one.
On the other hand, when transition 3
occurs, $|\rho(\phi^e)|=|\rho(\phi)|+1$ 
and the tree containing $e$ is fragmented into two trees where the new appearing root is at 
$e_-$.
Transitions of type 2 produce a rearranging in one of the tree. 
They leave invariant the cardinality of the set of roots but 
the location of the root in the modified tree is relocated at the vertex 
$e_-'=e^+$.

\begin{proposition}\label{InvMeasRootDyn}{\bf{(Invariance)}}
For all $q > 0$ the measure $w_q$ in \eqref{ForestMeasure} is invariant
for the dynamic defined by $\mathcal{G}$ in \eqref{ForestGen}.
\end{proposition}
\begin{proof}
We have to show that for any $\phi'\in \mathcal{F}$
\begin{equation}\label{InvCond}
\sum_{\phi\in \mathcal{F}}\sum_{e\in\mathcal{E} : \phi^e= \phi'}w_q(\phi)\gamma(\phi,e)=
\sum_{e\in\mathcal{E}} w_q(\phi')\gamma(\phi',e).
\end{equation}
For a given $\phi'\in \mathcal{F}$, we start by splitting the l.h.s. of \eqref{InvCond}
in the three terms corresponding to the different transitions
allowed by the dynamics in Definition \ref{RootDynamics} whenever $\gamma(\phi,e)>0$.
\begin{equation}
\begin{aligned}
\sum_{\phi\in\mathcal{F}}\sum_{e\in\mathcal{E} : \phi^e= \phi'}w_q(\phi)\gamma(\phi,e)&=
\sum_{\phi\subsetneq\phi'}\sum_{e\in\mathcal{E} : \phi^e= \phi'}w_q(\phi)\gamma(\phi,e)
+\sum_{\phi\supsetneq\phi'}\sum_{e\in\mathcal{E} : \phi^e= \phi'}w_q(\phi)\gamma(\phi,e)\\&
+\sum_{\phi: |\rho(\phi)|=|\rho(\phi')|}\sum_{e\in\mathcal{E} : \phi^e= \phi'}
w_q(\phi)\gamma(\phi,e)=I+II+III.
\end{aligned}
\end{equation}
We can rewrite
\begin{equation}
\begin{aligned}\label{I}
I&=\sum_{e\in\phi'}w_q(\phi'\setminus\{e\})\gamma(\phi'\setminus\{e\},e)=
\sum_{e\in\phi'}w_q(\phi')\gamma(\phi',e),
\end{aligned}
\end{equation}
\begin{equation}
\begin{aligned}\label{II}
II&=\sum_{e_-\in\rho(\phi'), e^+\notin \tau_{e_-}(\phi') }
w_q(\phi'\cup\{e\})\gamma(\phi'\cup\{e\},e)
=\sum_{e_-\in\rho(\phi'), e^+\notin \tau_{e_-}(\phi')}
w_q(\phi')\gamma(\phi',e),
\end{aligned}
\end{equation}
and, denoting, for each $e'$ such that $e'_- \in \rho(\phi')$ and
$e'^+ \in \tau_{e'_-}(\phi')$, by $e$ the unique bond
in the only one cycle of $\phi' \cup \{e'\}$ such that $e^+ \in \rho(\phi')$,
with $\phi = \phi' \setminus \{e\} \cup \{e'\}$,
\begin{equation}
\begin{aligned}\label{III}
III
&=\sum_{e'_-\in\rho(\phi'), e'^+\in \tau_{e'_-}(\phi')}
w_q(\phi)\gamma(\phi,e) 
=\sum_{e'_-\in\rho(\phi'), e'^+\in \tau_{e'_-}(\phi')}
w_q(\phi')\gamma(\phi',e') .
\end{aligned}
\end{equation}
Summing $I$, $II$ and $III$ together we then get (\ref{InvCond}).
\end{proof}

\par\noindent
{\bf Remark:} When $q = 0$ we recover the 
Anantharam and Tsoukas dynamics and the proof of the Markov tree theorem.
Indeed, the standard forest measure restricted to spanning trees is the invariant
measure of the dynamics. Starting with a single tree, its roots follows
a Markovian evolution with generator $L$, 
so that, at equilibrium, in the long time limit we have
\begin{equation} \label{nonnino}
	\mu(x) = \frac{
		\sum_{\phi : |\rho(\phi)| = 1}
		w(\phi) {\mathbbm 1}_{\{\rho(\phi) = \{x\}\}}
	}{
		\sum_{\phi : |\rho(\phi)| = 1}
		w(\phi) 
	},
\end{equation}
with $\mu$ being the stationary
distribution associated to $L$.

We can then describe the position
of the roots given the ``tree partition''
associated with $\Phi_q$.
For $\phi$ in ${\cal F}$ with roots $x_1$, \dots, $x_m$,
let us denote by ${\cal P}(\phi) = \{A_1, \dots, A_m\}$
the partition of ${\cal X}$ where each component $A_i$ is the set
of sites spanned by $\tau_{x_i}(\phi)$.
Since each piece of the partition comes with a special point corresponding to a root,
we call {\em rooted partition} the pair $({\cal P}(\phi), \rho(\phi))$ 
We note that for each $A$ in ${\cal P}(\phi)$,
the {\em restricted dynamics} with generator
\begin{equation}
	(L_A f)(x) = \sum_{y \in A} w(x, y) [f(y) - f(x)], 
	\quad f: A \rightarrow {\mathbb R},
	\quad x \in A,
\end{equation}
has only one irreducible component since each $x$ in $A$
is connected with the root. As a consequence
the restricted dynamics has a unique equilibrium measure which we call 
{\em restricted measure}~$\mu_A$.
Note that when $L$ has a {\em reversible} equilibrium $\mu = \mu_{\cal X}$,
then $\mu_A$ is nothing but the equilibrium measure $\mu$ conditioned on
$A$, i.e. $\mu_A = \mu(\cdot | A)$. 
\begin{proposition}\label{LocalEquilibria}{\bf(Roots in restricted equilibrium)}
Fix $m \in \{1,\ldots,n\}$, then 
\begin{equation}\label{localEquilibria}
{\mathbb P}\Big(\rho(\Phi_q)=\{x_1,\cdots,x_m\} \Big| \mathcal{P}(\Phi_q)=\{A_1,\ldots, A_m\}\Big)
=\prod_{i=1}^m \mu_{A_i}(x_i),
\end{equation}
for any partition $\{A_1,\ldots, A_m\}$ of $\mathcal{X}$ and any $x_i\in A_i$, for $i=1,\ldots,m.$
\end{proposition}
\begin{proof}
For each $i$ in $\{1, \dots, m\}$,
let us call ${\cal T}_i$ the set of rooted spanning trees of $A_i$.
For each $\tau_i$ in ${\cal T}_i$, define
\begin{equation}
	w_i(\tau_i) = \prod_{e \in \tau_i} w(e),
\end{equation}
and for $y_i$ in $A_i$, write $	\rho(\tau_i) = \{y_i\}$, if $y_i$ is the root of the tree $\tau_i \in{\cal T}_i$.
Compute 
\begin{equation}
	\begin{aligned}
		&{\mathbb P} \left(
			\rho(\Phi_q)=\{x_1,\cdots,x_m\}
			\Big|
			\mathcal{P}(\Phi_q)=\{A_1,\ldots, A_m\}
		\right)\\
		&\quad = \frac{
			{\mathbb P}\Big(\rho(\Phi_q)=\{x_1,\cdots,x_m\},\mathcal{P}(\Phi_q)=\{A_1,\ldots, A_m\}\Big)
		}{
			{\mathbb P}\Big(\mathcal{P}(\Phi_q)=\{A_1,\ldots, A_m\}\Big)
		}\\
		&\quad =\frac{
			q^m 
			\sum_{\tau_1 \in {\cal T}_1}
			\dots
			\sum_{\tau_m \in {\cal T}_m}
			\prod_{i = 1}^m w_i(\tau_i) {\mathbbm 1}_{\{\rho(\tau_i) = \{x_i\}\}}
		}{
			q^m 
			\sum_{\tau_1 \in {\cal T}_1}
			\dots
			\sum_{\tau_m \in {\cal T}_m}
			\prod_{i = 1}^m w_i(\tau_i)
		}\\
		&\quad = \prod_{i = 1}^m \frac{
			\sum_{\tau_i \in {\cal T}_i} w_i(\tau_i) {\mathbbm 1}_{\{\rho(\tau_i) = \{x_i\}\}}
		}{
			\sum_{\tau_i \in {\cal T}_i} w_i(\tau_i) 
		}\\
		&\quad = \prod_{i = 1}^m \mu_{A_i}(x_i)
	\end{aligned}
\end{equation}
where the last equality follows by \eqref{nonnino}
applied to the restricted dynamics.
\end{proof}

\begin{figure}[htbp]
	\caption{
		\label{crous}
		A rooted partition with 50 roots (at the center of red diamonds)
		sampled for the Metropolis random walk in a Brownian sheet potential
		on the $987 \times 610$ grid and with inverse temperature $\beta = .06$.
		Blue levels depend on the potential:
		the lower the potential, the darker the blue.
		We see that each root is distributed according to the restricted equilibrium of its own piece of the partition.
		See also the first two pictures in Figure~\ref{mimosa}.
	}
	\begin{center}
		\includegraphics[width=400pt]{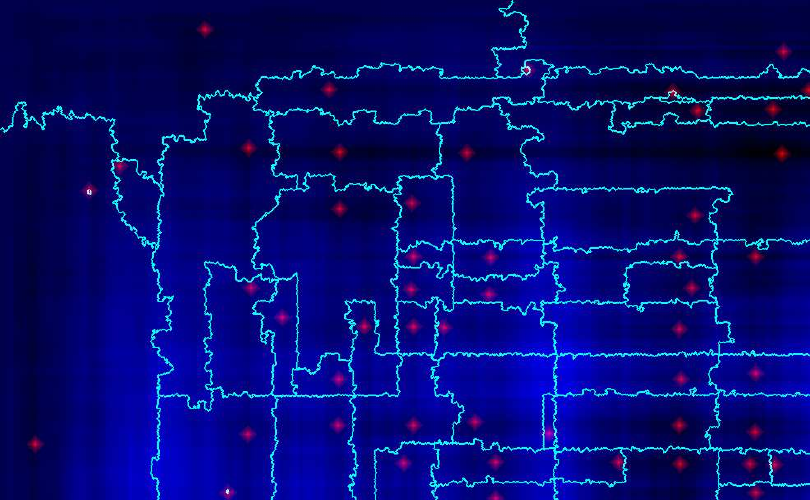}
	\end{center}
\end{figure}

\par\noindent
{\bf Remark:} 
When $X$ is reversible with respect to a measure $\mu$,
this gives a way to build the associated Gaussian free field
with mass $m = \sqrt{2q}$, that is the Gaussian process
$\xi = (\xi_x)_{x \in {\mathcal X}}$ with covariance matrix
\begin{equation}
	\Gamma = \left(
		\frac{G_q(x, y)}{\mu(y)}
	\right)_{x, y \in {\mathcal X}}
	= \left(
		\frac{K_q(x, y)}{q\mu(y)}
	\right)_{x, y \in {\mathcal X}}
	,
\end{equation}
by successive sampling of the standard measure $\nu_q$.
Start from independent centered random variables 
$\zeta_x$ with variance $\mu(x)$, $x \in {\mathcal X}$, sample $\Phi$
according to $\nu_q$, call $A(x)$ the set of vertices
of $\tau_x(\Phi)$  and set
\begin{equation}
	\tilde\xi_x
	= \frac{1}{\mu(A(x))} \sum_{y \in A(x)} \frac{\zeta_y}{\sqrt{q}}
	.
\end{equation}
Then the random field $\tilde\xi$ has zero mean and covariance matrix $\Gamma$
and the rescaled partial sums $\sum_{i=1}^{n} \tilde\xi^i / \sqrt{n}$,
with $\tilde\xi^1$, $\tilde\xi^2$, \dots independent
copies of $\tilde\xi$, 
converges in law to $\xi$ as $n$ goes to infinity.
Indeed, for each $x$ and $y$ in $\mathcal{X}$, $\tilde\xi_x$ and $\tilde\xi_y$
are centered and one computes
\begin{equation}
	{\mathbb E}\left[
	\tilde\xi_x \tilde\xi_y
	\right]
	= \sum_{A\subset\mathcal{X}} \mathbb{P}\left(
		A(x) = A(y) = A
	\right) \frac{1}{\mu(A)^2} \sum_{z \in A} \frac{\mu(z)}{q}
	= \sum_{A\ni x, y} \mathbb{P}\left(
		A(x) = A(y) = A
	\right) \frac{1}{q\mu(A)}
	.
\end{equation}
Since, following Wilson's algorithm,
\begin{align}
	K_q(x, y)
	&= \mathbb{P}\left(
		y \in \rho(\Phi_q),
		A(x) = A(y)
	\right)\\
	&= \sum_{A \ni x, y} \mathbb{P}\left( 
		A(x) = A(y) = A
	\right) \mathbb{P}\left(
		y \in \rho(\Phi_q) \Big|\,
		A(x) = A(y) = A
	\right)\\
	&= \sum_{A \ni x, y} \mathbb{P}\left( 
		A(x) = A(y) = A
	\right) \frac{\mu(y)}{\mu(A)}
	,
\end{align}		
we conclude 
\begin{equation}
	{\mathbb E}\left[
		\tilde\xi_x \tilde\xi_y
	\right]
	= \frac{K_q(x, y)}{q\mu(y)}
	.
\end{equation}

%%%%%%%%%%%%%%%%%%%%%%%%%%%%%%%%%%%%%%%%%%%%%%%%%%%%%%%%%%%%%%%%%%%%%%%%%%%%%%%%%%%%%%%%%%%%%%%%%%%%%%%%%%%%%%%%%%%%
\section{The edge process}\label{archi}

We assume in this section that 
$(q(x) :\: x \in {\cal X})$
is a non-zero collection
of {\em finite} killing rates.
The infinite rate case can simply be obtained
by computing the corresponding limit 
in the next assertions.

The next theorem is due to Chang~\cite{C13}.
We write it with our notation
and in Appendix~\ref{merlo}
we give Chang's proof in our context
for completion.
Let us first introduce some more notation.
Given $x$ on ${\cal X}$ and $e$ in ${\cal E}$
we call 
\begin{equation}
	J_Q^+(x, e) = G_Q(x, e_-) w(e)
\end{equation}
the expected number of crossings
of the (oriented) edge $e$ up to time $T_Q$.
We also define the net flow through $e$
starting from $x$ by
\begin{equation}
	J_Q(x, e) = J_Q^+(x, e) - J_Q^+(x, -e)
	\,.
\end{equation}

\begin{theorem}{\bf (Chang~\cite{C13})}\label{borgorosa}
	For any $A = \{e_1, \dots, e_k\} \subset {\cal E}$
	\begin{equation}\label{tenda}
		{\mathbb P}\left(
			A \subset \Phi_Q
		\right)
		= {\mathbb P}\left(
			e_1, e_2, \dots, e_k \in \Phi_Q
		\right)
		= {\rm det}_A {\cal K}_Q^+
	\end{equation}
	with 
	\begin{equation}
		{\cal K}_Q^+(e, e') = J_Q^+(e_-, e') - J_Q^+(e^+, e')
		\,, \quad e, e' \in {\cal E}
		\,.
	\end{equation}
	In addition,
	denoting by $\{\pm e_1, \dots, \pm e_k \in \Phi_Q\}$
	the event that for all $i \leq k$ either $e_i$ or $-e_i$
	belong to $\Phi_Q$, it holds
	\begin{equation}\label{anatra}
		{\mathbb P}\left(
			\pm e_1, \dots, \pm e_k \in \Phi_Q
	\right)
	= {\rm det}_A {\cal K}_Q
	\end{equation}	
	with 
	\begin{equation}
		{\cal K}_Q(e, e') = J_Q(e_-, e') - J_Q(e^+, e')
		\,, \quad e, e' \in {\cal E}
		\,.
	\end{equation}
\end{theorem}

\par\noindent{\bf Remark:}
In the symmetric case $w(x, y) = w(y, x) = 1$
for all $e = (x, y) \in {\cal E}$,
by choosing $q(x) = 1$ for all $x \in {\cal X}$
one obtains, as a consequence of~\eqref{anatra},
a proof of a tree-size biased version
of the negative edge correlation conjecture
for the uniform unrooted spanning forest
(see~\cite{K00, GW04}).
Recalling~\eqref{toussaint}, one has indeed,
for each $e$ and $e'$ in ${\cal E}$ 
and the associated unoriented edges $\bar e$ and $\bar e'$ in $\bar {\cal E}$,
\begin{equation}
	\begin{split}
		{\mathbb P}\left(
			\bar e, \bar e' \in \bar \Phi_1
		\right)
		& = {\cal K}_Q(e, e) {\cal K}_Q(e', e') - \left(
			{\cal K}_Q(e, e')
		\right)^2 \\
		& \leq {\cal K}_Q(e, e) {\cal K}_Q(e', e') 
		= {\mathbb P}\left(
			\bar e \in \bar \Phi_1
		\right) {\mathbb P}\left(
			\bar e' \in \bar \Phi_1
		\right)
		\,,
	\end{split}
\end{equation}
where the square in the first equation
is a consequence of the reversibility
of the Green's kernel with respect to the uniform measure.
But we were unable to deduce from this
a proof of the original conjecture.
				
\medskip\par
Consider the random forest $\Phi_{\lambda Q}$ for a given $\lambda$ and killing matrix $Q$.
Set $Z_Q(\lambda)=Z_{\lambda Q}$ and notice that, as $\lambda$ goes to 0,
$\Phi_{\lambda Q}$ converges in law to a random spanning tree $\tau_Q$ with distribution
\begin{equation}
	{\mathbb P}\left(\tau_Q=\tau\right)=\frac{w(\tau)q(\rho(\tau))}{Z_Q'(0)}
\end{equation}
where $\rho(\tau)$ stands for the root of the spanning tree $\tau$.
By~\eqref{nonnino} this is the law of a random tree obtained by running Wilson's algorithm
with root choosen with probability
\begin{equation}
	\nu(x)=\frac{q(x)\mu(x)}{\sum_{z}q(z)\mu(z)}
	\,,\quad x \in {\cal X},
\end{equation}
and used as the absorbing state.
Chang~\cite{C13} showed that $\tau_Q$
is also a determinantal process.
By computing the same limits in a different way,
we give here a different proof
and, more importantly, a different expression for the kernel,
allowing for an easy comparison with the reversible transfer current theorem.
For $x, y$ in ${\cal X}$ and $e$ in ${\cal E}$ 
we define $J_y^+(x, e)$ as the expected number of crossings
of the (oriented) edge $e$ by the process $X$ started from $x$ and stopped
at the hitting time of $y$, $T_y$.
We also define the net flow through $e$ by
\begin{equation}
	J_y(x, e) = J_y^+(x, e) - J_y^+(x, -e)
	\,.
\end{equation}

\begin{theorem}{\bf(Transfer current theorem)}\label{corvaro}
For any $A = \{e_1, \dots, e_k\} \subset {\cal E}$
	\begin{equation}\label{telefonino}
		{\mathbb P}\left(
			A \subset \tau_Q
		\right)
		= {\mathbb P}\left(
			e_1, e_2, \dots, e_k \in \tau_Q
		\right)
		= {\rm det}_A {\cal H}_Q^+
	\end{equation}
	with 
	\begin{equation}
		{\cal H}_Q^+(e, e') = \frac{
			\beta(e^+, e_-) J_{e^+}^+(e_-, e') - \beta(e_-, e^+) J_{e_-}^+(e^+, e')
		}{
			\beta(e^+, e_-) + \beta(e_-, e^+) 
		}
		\,, \quad e, e' \in {\cal E}
		\,,
	\end{equation}
	where
	\begin{equation}
		\beta(x, y) = {\mathbb E}_x\left[
			\int_0^{T_y} q(X(t)) dt
		\right]
		\,, \quad x, y \in {\cal X}
		\,.
	\end{equation}
	In addition,
	\begin{equation}\label{cavallo}
		{\mathbb P}\left(
			\pm e_1, \dots, \pm e_k \in \tau_Q
	\right)
	= {\rm det}_A {\cal H}_Q
	\end{equation}	
	with 
	\begin{equation}
		{\cal H}_Q(e, e') = \frac{
			\beta(e^+, e_-) J_{e^+}(e_-, e') + \beta(e_-, e^+) J_{e_-}(e^+, -e')
		}{
			\beta(e^+, e_-) + \beta(e_-, e^+) 
		}
		\,, \quad e, e' \in {\cal E}
		\,.
	\end{equation}
\end{theorem}

\par\noindent
{\bf Remark:} In the reversible case 
one has, for any $x, y$ in ${\cal X}$ and $e$ in ${\cal E}$,
$J_y(x, e) = J_x(y, -e)$, so that
\begin{equation}
	{\cal H}(e, e') = J_{e^+}(e_-, e')
	\,,
\end{equation}
and we recover Burton and Pemantle's theorem~\cite{BP93}.
The difference 
\begin{equation}\label{penna}
	J_y(x, e) - J_x(y, -e)
	= J_y(x, e) + J_x(y, e)
\end{equation}
is indeed the net flow 
through the edge $e$ during the commutation between $x$ and $y$,
i.e,  along the process started from $x$ and stopped at its first
return to $x$ after $T_y$.
Since, by reversibility, each such path appears in the corresponding expectation
with the same weight as its reversed path, and it contributes to the net flow
with the opposite sign, the difference~\eqref{penna} is equal to zero.

\begin{proof}
We simply compute the limit of $K_{\lambda Q}^+(e, e')$ as $\lambda$ goes to 0.
To this end we set $e = (x, y)$ and, for any $t_2 \geq t_1 \geq 0$,
we denote by $\mathcal{C}^{+}_{e'}[t_1, t_2]$ the (random) number of crossing of the edge $e'$ (from $e'_-$ to $e'^+$)
in the time interval $[t_1, t_2]$.
Hence $J_Q^+(e_-, e')=\mathbb{E}_{x}\bigl[\mathcal{C}^{+}_{e'}[0, T_Q]\bigr]$
and 
\begin{equation}
	\begin{split}
		\lim_{\lambda\to 0} {\cal K}_{\lambda Q}^+(e, e')
		&=\lim_{\lambda\to 0}\left\{
			J_{\lambda Q}^+(e_-, e')-J_{\lambda Q}^+(e_+, e')
		\right\}\\
		& =\lim_{\lambda\to 0} \left\{
			\mathbb{E}_{x}\left[
				\mathcal{C}^{+}_{e'}[0, T_{\lambda Q}]
			\right]
			- \mathbb{E}_{y}\left[
				\mathcal{C}^{+}_{e'}[0, T_{\lambda Q}]
			\right]
		\right\}
		\,.
	\end{split}
\end{equation}
Set $\sigma_0 = 0$, $\sigma_1 = T_{y}$, define recursively
\begin{equation}
	\begin{split}
		& \sigma_{2k} = \inf\{t \geq \sigma_{2k-1} :\: X(t) = x\}
		\,, \\
		& \sigma_{2k+1} = \inf\{t \geq \sigma_{2k} :\: X(t) = y\}
		\,,
	\end{split}
\end{equation}
and notice that 
\begin{equation}
	\begin{split}	
		& \mathbb{E}_{x}\left[
			\mathcal{C}^{+}_{e'}[0, T_{\lambda Q}]
		\right] \\
		& \quad = \sum_{k\geq 0} \left\{
				\mathbb{E}_{x}\left[
					\mathcal{C}^{+}_{e'}\left[
						\sigma_{2k} \wedge T_{\lambda Q}, \sigma_{2k+1} \wedge T_{\lambda Q}
					\right]
				\right]
				+ \mathbb{E}_{x}\left[
					\mathcal{C}^{+}_{e'}\left[
						\sigma_{2k+1} \wedge T_{\lambda Q}, \sigma_{2k+2} \wedge T_{\lambda Q}
					\right]
				\right]
		 \right\}
		\,.
	\end{split}
\end{equation}
With
\begin{equation} \label{maia}
	\xi_{x, y}(\lambda) = {\mathbb P}_x\left(
		T_y < T_{\lambda Q}
	\right)
	= {\mathbb E}_x\left[
		\exp\left\{
			-\lambda \int_0^{T_y} q(X(t)) dt
		\right\}
	\right]
\end{equation}
and 
\begin{equation}
	J^+_{\lambda Q, y}(x, e') = {\mathbb E}_x\left[
		{\cal C}^+_{e'}[0, T_y \wedge T_{\lambda Q}]
	\right]
\end{equation}
we get
\begin{equation}
	\begin{split}	
		\mathbb{E}_{x}\left[
			\mathcal{C}^{+}_{e'}[0, T_{\lambda Q}]
		\right] 
		& = \sum_{k\geq 0} \left(
			\xi_{x, y}(\lambda) \xi_{y, x}(\lambda)
		\right)^k \left\{
			J^+_{\lambda Q, y}(x, e')
			+ \xi_{x, y}(\lambda) J^+_{\lambda Q, x}(y, e')
		 \right\} \\
		& = \frac{
			J^+_{\lambda Q, y}(x, e')
			+ \xi_{x, y}(\lambda) J^+_{\lambda Q, x}(y, e')
		}{
			1 - \xi_{x, y}(\lambda) \xi_{y, x}(\lambda)
		}
	\end{split}
\end{equation}
and
\begin{equation}
	\begin{split}
		{\cal K}_{\lambda Q}^+(e, e')
		& = \mathbb{E}_{x}\left[
			\mathcal{C}^{+}_{e'}[0, T_{\lambda Q}]
		\right]
		- \mathbb{E}_{y}\left[
			\mathcal{C}^{+}_{e'}[0, T_{\lambda Q}]
		\right] \\
		& = \frac{
			\left(
				1 - \xi_{y, x}(\lambda)
			\right) J^+_{\lambda Q, y}(x, e')
			- \left(
				1 - \xi_{x, y}(\lambda)
			\right) J^+_{\lambda Q, x}(y, e')
		}{
			1 - \xi_{x, y}(\lambda) \xi_{y, x}(\lambda)
		}
		\,.
	\end{split}
\end{equation}
By differentiating in $\lambda$ equation~\eqref{maia}
one gets
\begin{equation}
	\xi_{x, y}(\lambda) = 1 - \lambda {\mathbb E}_x\left[
		\int_0^{T_y} q(X(t)) dt
	\right]
	+ o(\lambda)
\end{equation}
and 
\begin{equation}
	\begin{split}
		\lim_{\lambda \rightarrow 0} {\cal K}_{\lambda Q}^+(e, e')
		& = \frac{
			\beta(y, x) J^+_{y}(x, e')
			- \beta(x, y) J^+_{x}(y, e')
		}{
			\beta(y, x) + \beta(x, y)
		}
		\,.
	\end{split}
\end{equation}
This proves~\eqref{telefonino} and~\eqref{cavallo}.
\end{proof}

%%%%%%%%%%%%%%%%%%%%%%%%%%%%%%%%%%%%%%%%%%%%%%%%%%%%%%%%%%%%%%%%%%%%%%%%%%%%%%%%%%%%%%%%%%%%%%%%%%%%%%%%%%%%%%%%%%%%
\appendix 

\section{Schur complement and trace process}\label{SCHUR}
Assume we have a Markov process $Y$ on a finite state space $\mathcal{Y}$, 
with generator $\mathcal{L}$ given by
\begin{equation}
(\mathcal{L}f)(y)=\sum_{z\in \cal Y}\alpha(y,z)[f(z)-f(y)],\quad y\in \mathcal{Y}, 
\end{equation}
with  $f:\mathcal{Y}\rightarrow \mathbb{R}$ arbitrary and $\{\alpha(y,z)\in [0, +\infty) :\: (y,z)\in \mathcal{Y}\times \mathcal{Y}\}$ 
a given collection of non-negative and finite transition rates.
With
\begin{equation}
	\bar \alpha = \max_{y \in {\cal Y}} \alpha(y)
	= \max_{y \in {\cal Y}} \sum_{z \neq y} \alpha(y, z)
	< + \infty \,,
\end{equation}
we define a discrete skeleton $\hat Y$ of $Y$
with transition matrix $\hat P$ such that
\begin{equation}
	{\cal L} = \bar \alpha (\hat P - \mathbbm{1}_{\cal Y}) \,.
\end{equation}
In other words, $Y$ can be build from $\hat Y$
by updating the position of the process
after independant exponential times of paramater $\bar \alpha$
according to the position sequence of $\hat Y$. 

Fix a subset $A\subset \mathcal{Y}$ and define a new Markov chain $\hat{Y}^A$ with state space $A$ obtained as the trace of the process $\hat Y$ on $A$.
In other words, $\hat{Y}^A$ is the random walk with transition matrix $\hat{P}^A$
with entries 
\begin{equation}\label{trace}
\hat{P}^A(x,y)=P_x(\hat{Y}(\hat{T}^+_A)=y), \quad \text{ for } x,y \in A,
\end{equation}
where $\hat{T}^+_A$ denotes the first return time in $A$ of the chain $\hat Y$.

Back to the continuous-time setting, denote by $Y^A$ be the continuous-time version of $\hat{Y}^A$ 
with jump times given by exponential random variables of parameter 
\begin{equation}
\bar\alpha=\max_{y\in {\cal Y}}\alpha(y)<\infty,
\end{equation}
i.e. the process with generator
\begin{equation}\label{contVSdisc}
\mathcal{L}^A=\overline{\alpha}(\hat{P}^A-\mathbbm{1}_A).
\end{equation}
Equivalently, $Y^A$ is the trace of the Markov process $Y$ on $A$, namely,
the process obtained by following the trajectory of $Y$
at infinite velocity outside $A$ and without speeding up inside $A$.

\begin{proposition}\label{SchurProb}{\bf{(Schur complement and trace process)}}
Given the Markov process $Y$ on $\mathcal{Y}$ with generator $\mathcal{L}$,
fix a subset  $A\subset \mathcal{Y}$.
Let $\mathcal{L}^A$ be the generator of the Markov process $Y^A$ obtained as the trace of the process $Y$ on $A$.
Then, $\mathcal{L}^A$ is the Schur complement of $[{\cal L}]_{A^c}$ in $\mathcal{L}$, i.e.
\begin{equation}\label{schur2}
\mathcal{L}^A=[\mathcal{L}]_A-[\mathcal{L}]_{A,A^c}[\mathcal{L}]^{-1}_{A^c}[\mathcal{L}]_{A^c,A} \,,
\end{equation}
where $[{\cal L}]_{A,A^c}$ and $[{\cal L}]_{A^c, A}$ are the operators obtained from the matrix representation of ${\cal L}$,
by keeping only those rates from $A$ to $A^c$ and $A^c$ to $A$, respectively.
\end{proposition}
\begin{proof}
Denote by $[\hat{P}]_A$, the sub-Markovian matrix $\hat{P}$ restricted to $A$.
Note that $[\hat{P}]_A$ is different from $\hat{P}^A$.
Due to \eqref{trace}, for any $x,y\in A$, we can write 

\begin{equation}\begin{aligned}\label{S1}
\hat{P}^A(x,y)&=\hat{P}(x,y)+\sum_{z,z'\in A^c}\hat{P}(x,z)\left(\sum_{k\geq 0}[\hat{P}]_{A^c}^k(z,z')\right)\hat{P}(z',y)
\\&
=[\hat{P}]_{A}(x,y)+\sum_{z,z'\in A^c}\hat{P}(x,z)\left(\mathbbm{1}_{A^c}-[\hat{P}]_{A^c}\right)^{-1}(z,z')\hat{P}(z',y).
\end{aligned}
\end{equation}

Subtracting $\mathbbm{1}_{A}$ on both side of \eqref{S1}, we obtain that
\begin{equation}\begin{aligned}\label{S2}
\hat{P}^A-\mathbbm{1}_{A}
&= [\hat{P}-\mathbbm{1}]_A -  [\hat{P}-\mathbbm{1}]_{A,A^c}\left([\hat{P}]_{A^c}-\mathbbm{1}_{A^c}\right)^{-1}[\hat{P}-\mathbbm{1}]_{A^c,A}.
\end{aligned}
\end{equation}
We then get our result by multiplying both sides by $\bar \alpha$.
\end{proof}

When ${\cal Y}$ contains an absorbing set $B$ 
and $A \subset B^c$ we can do the same
computation with the sub-Markovian generator $[{\cal L}]_{B^c}$
in place of ${\cal L}$.
For any $x$ and $y$ in $A$ the mean local time
in $y$ starting from $x$ and before hitting $B$ is the same for $Y$
and the trace process $Y^A$, i.e.
\begin{equation}
	G_B(x, y) = G_B^A(x, y),
\end{equation}
that is
\begin{equation} \label{cubo}
	\left[
		\left(
			[{\cal L}]_{B^c}
		\right)^{-1}
	\right]_A
	= \left(
		\left(
			[{\cal L}]_{B^c}
		\right)^A
	\right)^{-1}
	\,,
\end{equation}
or, to be more concise,  $[[{\cal L}]_{B^c}^{-1}]_A = ([{\cal L}]_{B^c}^A)^{-1}$.
More generally, one has the following
definition and properties.
\begin{definition}\label{Schur}{\bf{(Schur complement)}}
Let $\mathcal{M} = (\mathcal{M}(x, y))_{x, y \in {\cal X}}$ be an $n \times n$ matrix 
written as $2\times 2$ block matrix 
\begin{equation}\label{Blocks}
\mathcal{M} =\left[
	\begin{array}{cc}
		\mathcal{A} & \mathcal{B}\\ 
		\mathcal{C} & \mathcal{D} 
		\end{array}
	\right] \,,
\end{equation}
where ${\cal A} = [{\cal M}]_A = ({\cal M}(x, y))_{x, y \in A}$,
${\cal B} = [{\cal M}]_{A, A^c} = ({\cal M}(x, y))_{x \in A, y \not \in A}$,
${\cal C} = [{\cal M}]_{A^c, A} = ({\cal M}(x, y))_{x \not \in A, y \in A}$ and
${\cal D} = [{\cal M}]_{A^c} = ({\cal M}(x, y))_{x, y \not \in A}$
for some $A \subset {\cal X}$.
The Schur complement of ${\cal D}$ in ${\cal M}$ is defined as the matrix
\begin{equation}\label{SchurD}
	S_{{\cal M}}({\cal D}) = {\cal A} - {\cal B} {\cal D}^{-1} {\cal C} \,.
\end{equation}
\end{definition}
One can then check
\begin{equation}
	\left[
		\begin{array}{cc}
			{\cal A} & {\cal B} \\
			{\cal C} & {\cal D}
		\end{array}
	\right] 
	= \left[
		\begin{array}{cc}
			\mathbbm{1}_A & {\cal B} {\cal D}^{-1} \\
			0 & \mathbbm{1}_{A^c} 
		\end{array}
	\right] \left[
		\begin{array}{cc}
			S_{\cal M}({\cal D}) & 0 \\
			0 & {\cal D}
		\end{array}
	\right] \left[
		\begin{array}{cc}
			\mathbbm{1}_A & 0 \\
			{\cal D}^{-1}{\cal C} & \mathbbm{1}_{A^c}
		\end{array}
	\right] \,.
\end{equation} 
It follows
\begin{equation}\label{detM}
{\rm det}({\cal M})={\rm det}({\cal D}){\rm det}\left(S_{\cal M}({\cal D})\right)
\end{equation}
and, as a generalization of (\ref{cubo}),
\begin{equation}\label{inverse}
{\cal M}^{-1}=\left[\begin{array}{cc}
S_{{\cal M}}({\cal D})^{-1} & -S_{{\cal M}}({\cal D}){\cal B}{\cal D}^{-1}\\ 
-{\cal D}^{-1}{\cal C}S_{{\cal M}}({\cal D}) & {\cal D}^{-1}+{\cal D}^{-1}{\cal C}S_{{\cal M}}({\cal D}){\cal B}{\cal D}^{-1} 
\end{array}\right].
\end{equation}
In particular, from (\ref{inverse}) and (\ref{detM}),
\begin{equation}\label{detM-1}
	{\rm det}_A\left(
		{\cal M}^{-1}
	\right) = {\rm det}\left(
		S_{\cal M}({\cal D})^{-1}
	\right)
	= {\rm det}(S_{\cal M}({\cal D}))^{-1}
	= \frac{
		{\rm det}_{A^c}({\cal M})
	}{
		{\rm det}({\cal M})
	}
\end{equation}
This relation is used in the algebraic proof of Theorem \ref{DeterminantalRoots}.

%%%%%%%%%%%%%%%%%%%%%%%%%%%%%%%%%%%%%%%%%%%%%%%%%%%%%%%%%%%%%%%%%%%%%%%%%%%%%%%%%%%%%%%%%%%%%%%%%%%%%%%%%%%%%%%%%%%%

\section{Lemma on hitting distributions and times}
\label{Wentzell}
In this section we use our forest measure analysis
to prove two formulas on the hitting distribution and the expectation of hitting times of a given subset of the given graph.
This result is originally due to Freidlin and Wentzell, see Lemmas 3.2 and 3.3 in~\cite{FW98}. 

\begin{lemma}\label{WentzellThm}{\bf{(Freidlin and Wentzell~\cite{FW98})}}
Fix a non-empty subset $R$ of $\mathcal{X}$. Recall the notation in \eqref{Rmeasures}. 
Consider the Markov process $X$ on $\mathcal{X}$ identified by \eqref{MarkovGenerator}, and let $T_R$ be the hitting time of the set $R$.
Then, for any $x$ and $y$ in ${\cal X}$, 
\begin{equation}\label{WilsonAbsorbingSet}
P_x(X(T_{R})=y)
= \frac{1}{Z_{R}(0)}
\sum_{
	\stackrel{
		\scriptstyle \phi:\rho(\phi)=R
	}{
		\scriptstyle \rho(\tau_x(\phi))=\{y\}
	}
}
w(\phi)
\end{equation}
and
\begin{equation}\label{WentzellFormula}
E_x[T_R] 
= \frac{1}{Z_R(0)}
\sum_{z \not\in R}
\sum_{
	\stackrel{
		\scriptstyle \phi:\rho(\phi)=R\cup\{z\}
	}{
		\scriptstyle \rho(\tau_x(\phi))=\{z\}
	}
}
w(\phi).
\end{equation}
\end{lemma}
\begin{proof}
Consider the extended space $\bar{\mathcal{X}}$ and the extended Markov process $\bar{X}$ as in \eqref{ExtendedMarkovGenerator}, 
with the killing rates as in \eqref{KillingR} for $q = 0$.
Equation (\ref{WilsonAbsorbingSet}) is simply obtained
by considering Wilson's algorithm started from $x$.

To prove (\ref{WentzellFormula})
we will use the discrete skeleton
of the absorbed process of Section \eqref{loops} with $\mathcal{Y}=\bar{\mathcal{X}}$ and $B=R$.
Let
\begin{equation}\label{GreensR}
G_R(x,y)= E_x[\ell_y(T_R)] \quad\text{and}\quad \hat{G}_R(x,y)= E_x[\ell_y(\hat{T}_R)], \quad x, y \in {\cal X},
\end{equation}
be the continuous and the discrete time Green's functions before hitting the set $R$, respectively.

Since $E_x[T_R]=\sum_{z\notin R}G_R(x,z)$, it suffices to show that, for all $z \not\in R$,
\begin{equation}\label{goal}
G_R(x,z)
= \frac{1}{Z_R(0)}
\sum_{
	\stackrel{
		\scriptstyle \phi:\rho(\phi)=R\cup\{z\}
	}{
		\scriptstyle \rho(\tau_x(\phi))=\{z\}
	}
}
w(\phi)
\,.
\end{equation}
Since
$\hat{G}_R(x,z)
= P_x(\hat T_z < \hat T_R) \hat G_R(z, z)
=P_x(\hat{T}_z<\hat{T}_R)/P_z(\hat{T}^+_z>\hat{T}_R),$ with $\hat{T}^+_z$
being the return time to $z$,
\begin{equation}
	\begin{aligned} \label{PotentialGreen2}
		G_R(x,z)
		& = \frac{1}{\bar \alpha} \hat{G}_R(x,z)
		= \frac{P_x(\hat T_z < \hat T_R)}{\bar \alpha P_z(\hat T^+_z > \hat T_R)}
		= \frac{P_x(\hat T_z < \hat T_R)}{\bar \alpha \sum_{y\neq z} \hat P(z,y) P_y(\hat T_z > \hat T_R)} \\
		& = \frac{P_x(T_z < T_R)}{\sum_{y \neq z} \alpha(z,y) \left[
			1-P_y(T_z < T_R)
		\right]}
		\,.
	\end{aligned}
\end{equation}
Observe that $P_x(T_z < T_R) = P_x(X(T_{R \cup \{z\}}) = z)$
for any $z \in R^c$, thus using \eqref{WilsonAbsorbingSet}:
\begin{equation}
	\begin{aligned}
		G_R(x,z)
		& = \frac{P_x(X(T_{R \cup \{z\}})=z)}{\sum_{y\neq z}\left[1-P_y(X(T_{R \cup \{z\}})=z)\right] \alpha(z,y)}\\
		&= \frac{
			\frac{1}{Z_{R \cup \{z\}}(0)}
			\sum_\phi \mathbbm{1}_{
				\left\{
					\rho(\phi) = R \cup \{z\},\,
					\rho(\tau_x(\phi)) = \{z\} 
				\right\}
			} w(\phi) 
		}{
			\sum_{y \neq z}
			\frac{1}{Z_{R\cup\{z\}}(0)}
			\sum_\phi \mathbbm{1}_{
				\left\{
					\rho(\phi) = R \cup \{z\},\,
					\rho(\tau_y(\phi)) \neq \{z\}
				\right\}
			} w(\phi) \alpha(z,y)
		} \\
		&= \frac{1}{Z_R(0)}
		\sum_{
			\stackrel{
				\scriptstyle \phi:\rho(\phi)=R\cup\{z\}
			}{
				\scriptstyle \rho(\tau_x(\phi))=\{z\}
			}
		}
		w(\phi).
	\end{aligned}
\end{equation}
\end{proof}

%%%%%%%%%%%%%%%%%%%%%%%%%%%%%%%%%%%%%%%%%%%%%%%%%%%%%%%%%%%%%%%%%%%%%%%%%%%%%%%%%%%%%%%%%%%%%%%%%%%%%%%%%%%%%%%%%%%%

\section{Divided differences}\label{DividedDiff}
In this appendix, we recall three equivalent definitions of the notion of divided differences of a real function. 
We further give a lemma due to Micchelli and Willoughby
for which we provide an alternative elementary proof which plays with these different definitions.
This result is used in Section \ref{MW}.

\begin{definition}{\bf(Divided differences 1)}
We call divided difference of a function $f$ at the distinct points $x_0$, $x_1$, \dots, $x_k$, the quantity $f[x_0,\cdots,x_k]$ recursively defined via
\begin{equation}\label{DivDiff1}
f[x_0,\cdots,x_k]= \frac{f[x_1,\cdots,x_k]-f[x_0,\cdots,x_{k-1}]}{x_k-x_0},
\end{equation}
with 
\begin{equation}
	f[x_i]=f(x_i)
	\,.
\end{equation}
\end{definition}
From this definition, we see that the divided differences at $k$ points of a function $f$ can be seen as a sort of $k$-th discrete derivative.
One then show by induction
\begin{definition}{\bf(Divided differences 2)}
For any function $f$ and distinct points $x_0$, $x_1$,~\dots, $x_{k - 1}$, $x_k$,
\begin{equation}\label{DivDiff2}
f[x_0,\cdots,x_k]= \sum_{i=0}^k\frac{f(x_i)}{\prod_{j\neq i}(x_i-x_j)}.
\end{equation}
\end{definition}
Note in this second definition, that $f[x_0,\cdots,x_k]$ is independent of the order of the $x_i$'s.
From (\ref{DivDiff2}) one can then check
\begin{definition}{\bf(Divided differences 3)} \label{papelito}
For any function $f$ and distinct points $x_0$, $x_1$,~\dots, $x_{k - 1}$, $x_k$,
\begin{equation}\label{DivDiff3}\begin{aligned}
Q(x)=&f[x_0]+f[x_0,x_1](x-x_0)+f[x_0,x_1,x_2](x-x_0)(x-x_1)\\&+\ldots+f[x_0,\cdots,x_k](x-x_0)\ldots(x-x_{k-1})
\end{aligned}
\end{equation}
is the unique polynomial of degree $k$ with $Q(x_i)=f(x_i)$, for $i=0,\cdots,k.$
\end{definition}

\begin{lemma}\label{DivDif}
{\bf (Micchelli and Willoughby \cite{MW79})}
Consider a polynomial of degree $n$ of the form 
\begin{equation}
	f(x)=\prod_{i=0}^{n - 1}(x-\alpha_i)
	\,, 
\end{equation}
with $n$ distinct real zeros $\alpha_i$ in decreasing order:
$\alpha_0 > \alpha_1 > \cdots > \alpha_{n - 1}$.
Let $\beta_0 > \beta_1 > \cdots > \beta_N$ with $N \geq n$
and such that
\begin{equation}\label{beta}
\beta_i\geq \alpha_i, \quad \mbox{for all $i < n$}.
\end{equation}
Then, for any $k\leq N$, 
\begin{equation}
	f[\beta_0,\beta_{1},\dots,\beta_{k}]\geq 0
	\,.
\end{equation}
\end{lemma}
\begin{proof}
We prove the following statement by induction on $r = n - k$:
\begin{equation} \label{chiocciola}
	\mbox{
		``For any $\beta_0 > \beta_1 > \cdots > \beta_N$
		satisfying (\ref{beta}),
		$f[\beta_0, \dots, \beta_k] \geq 0$.''
	}
\end{equation}
Since $f$ is a polynomial of degree $n$,
(\ref{chiocciola}) follows from Definition \ref{DivDiff3}
as soon as $r < 0$.
Also, since the dominant coefficient of $f$ is~$1$,
the same argument shows $f[\beta_0, \dots, \beta_n] = 1$
and the claim holds for $r = 0$.
Fix now $r > 0$, i.e. $k < n$, 
and $\beta_0 > \cdots > \beta_N$ satisfying (\ref{beta}).
If $\beta_0 \neq \alpha_0$ then
\begin{align}
	\frac{
		f[\beta_0, \alpha_1, \dots, \alpha_k] - f[\alpha_0, \alpha_1, \dots, \alpha_k]
	}{
		\beta_0 - \alpha_0
	}
	& = \frac{
		f[\beta_0, \alpha_1, \dots, \alpha_k] - f[\alpha_1, \dots, \alpha_k, \alpha_0]
	}{
		\beta_0 - \alpha_0
	} \label{crema} \\
	& = f[\beta_0, \alpha_1, \dots, \alpha_k, \alpha_0]\\
	& = f[\beta_0, \alpha_0, \alpha_1, \dots, \alpha_k] 
	. \label{nuvola}
\end{align} 
By Definition \ref{DivDiff2} we have $f[\alpha_0, \dots, \alpha_k] = 0$
and the numerator in the l.h.s. of (\ref{crema})
is merely equal to $f[\beta_0, \alpha_1, \dots, \alpha_k]$.
The denominator is positive from (\ref{beta}) and so is the r.h.s.
of (\ref{nuvola}) by induction hypothesis.
It follows that
\begin{equation}
	f[\beta_0, \alpha_1, \dots, \alpha_k] \geq 0
\end{equation}
and the same is true when $\beta_0 = \alpha_0$.
If $\beta_1 \neq \alpha_1$, 
we compute
\begin{equation}
	\frac{
		f[\beta_0, \beta_1, \alpha_2, \dots, \alpha_k] - f[\beta_0, \alpha_1, \dots, \alpha_k]
	}{
		\beta_1 - \alpha_1
	} = f[\beta_0, \beta_1, \alpha_1, \dots, \alpha_k]
	,
\end{equation}
get then in the same way 
\begin{equation}
	f[\beta_0, \beta_1, \alpha_2, \dots, \alpha_k] \geq 0
\end{equation}
and the same is true when $\beta_1 = \alpha_1$.
Proceeding similarly we eventually
obtain that 
\begin{equation}
	f[\beta_0, \dots, \beta_k] \geq 0.
\end{equation}
\end{proof}

%%%%%%%%%%%%%%%%%%%%%%%%%%%%%%%%%%%%%%%%%%%%%%%%%%%%%%%%%%%%%%%%%%%%%%%%%%%%%%%%%%%%%%%%%%%%%%%
\section{Proof of Theorem~\ref{borgorosa}} \label{merlo}

To prove~\eqref{tenda} we first note
that the left-hand side is equal to zero
as soon as there are two distinct edges $e$ and $e'$
in $A$ such that $e_- = e'_-$, and so is the right-hand side,
since one gets two proportional columns in the corresponding restricted matrix.
We then assume that for each $x$ in 
\begin{equation}
	A_- = \left\{
		x \in {\cal X} :\: \exists e \in A, e_- = x
	\right\}
\end{equation}
there is a unique $e$ in $A$ such that $x = e_-$.
Let us compute 
\begin{equation}\label{finestra}
	{\mathbb P}(A \subset \Phi_Q)
	= \frac{1}{Z_Q} \sum_{\phi \ni e_1, \dots, e_k} w_Q(\phi)
	\,.
\end{equation}
By Theorem~\ref{PartitionThm}
$Z_Q = \det(Q - L)$
and the sum appearing in~\eqref{finestra}
is equal to $\det(\tilde Q - \tilde L)$,
with $\tilde L$ the Markovian generator obtain from $L$
by setting to zero all rates $w(x,y)$ such that $x \in A_-$
and $(x, y) \neq e \in A$ with $e_- = x$, 
and with $\tilde Q$ obtained from $Q$ just by setting to zero
all the rates $q(x)$ for $x$ in $A_-$.

This gives
\begin{equation}
	{\mathbb P}(A \subset \Phi_Q)
	= \frac{
		\det(\tilde Q - \tilde L) 
	}{
		\det(Q - L)
	}
	= \det\left(
		(\tilde Q - \tilde L) (Q - L)^{-1}
	\right)
	= \det\left(
		(\tilde Q - \tilde L) G_Q
	\right)
\end{equation}
and we can compute row by row this matrix product.
Observing that, for each $e \in A$, $-w(e)$ has to be the diagonal term
in each line $x = e_-$ of the matrix representation of $\tilde L$,
we obtain
\begin{equation}
	(\tilde Q - \tilde L) G_Q(x, \cdot) 
	= \left\{
		\begin{array}{ll}
			{\mathbbm 1}_{\{x = \cdot\}} & \mbox{if $x \not \in A_-$} \\
			w(e) G_Q(e_-, \cdot) - w(e) G_Q(e^+, \cdot) & \mbox{if $x = e_-$ for $e \in A$}
		\end{array}
	\right.
\end{equation}
We then get~\eqref{tenda} by using the bloc triangular structure
of this product and redistributing the terms of the product $\prod_{e \in A} w(e)$ 
in the computed determinant.

To prove~\eqref{anatra} we introduce the matrices
\begin{equation}
	M = (K_Q^+(e, e'))_{e, e' \in A}
	\quad {\rm and} \quad
	N = (K_Q^+(e, -e'))_{e, e' \in A}
	\,,
\end{equation}
observe that, for all $e$ and $e'$ in ${\cal E}$,
\begin{equation}
	K^+_Q(-e, e') = - K^+_Q(e, e')
\end{equation}
and note that
\begin{equation}
	{\mathbb P}(\pm e_1, \dots, \pm e_k \in \Phi_Q)
	= \sum_{\sigma \in \{-, +\}^k} {\mathbb P}(\sigma _1 e_1, \dots, \sigma _k e_k \in \Phi_Q)
\end{equation}
is the coefficient of degree $k$ in the polynomial in $\lambda$
\begin{equation}
	\det\left[
		\lambda + K_Q^+
	\right]_{A \cup -A}
	= \left|
		\begin{array}{c|c}
			\lambda I + M & N \\
			\hline
			- M & \lambda I - N
		\end{array} 
	\right|
\end{equation}
where $I$ is the $k \times k$ identity matrix.
Elementary operations on rows and columns give then
\begin{equation}
	\det\left[
		\lambda + K_Q^+
	\right]_{A \cup -A}
	= \left|
		\begin{array}{c|c}
			\lambda I + M & N \\
			\hline
			\lambda I & \lambda I 
		\end{array} 
	\right|
	= \left|
		\begin{array}{c|c}
			\lambda I + M - N & N \\
			\hline
			0 & \lambda I 
		\end{array} 
	\right|
	= \lambda^k \det(\lambda + M - N)
	\,.
\end{equation}
The coefficient of degree $k$ in this polynomial in $\lambda$
is then
\begin{equation}
	{\mathbb P}(\pm e_1, \dots, \pm e_k \in \Phi_Q)
	= \det(M - N)
	= {\rm det}_A {\cal K}_Q
	\,.
\end{equation}	
\hfill$\Box$

\bigskip\par\noindent
{\bf Acknowledgements:}
A.G. thanks Laurent Tournier
for the many fruitful discussions on the matter,
and, in particular, the nice independence argument
we presented in Section~\ref{Background}.
He also thanks Pablo Ferrari, In\`es Armendariz,
Pablo Groisman, Leonardo Rolla and Matthieu Jonckheere
for their hospitality at Buenos Aires where all
this work started.
L.A. has been supported by DFG SPP Priority Programme 1590
Probabilistic Structures in Evolution and by NWO Gravitation Grant 024.002.003-NETWORKS.

%%%%%%%%%%%%%%%%%%%

\end{document}